\pdfoutput=1
\documentclass[11pt]{article}

\usepackage{color}
\usepackage{wrapfig,fullpage}
\usepackage{url}
\usepackage{graphicx}
\usepackage{paralist}
\usepackage{amsmath,amssymb,amsthm}
\usepackage{latexsym}
\usepackage{amsfonts}
\usepackage[numbers]{natbib}
\usepackage[ruled,linesnumbered]{algorithm2e}
\usepackage{hyperref}

\definecolor{Red}{rgb}{1,0,0}
\definecolor{Blue}{rgb}{0,0,1}
\definecolor{Olive}{rgb}{0.41,0.55,0.13}
\definecolor{Yarok}{rgb}{0,0.5,0}
\definecolor{Green}{rgb}{0,1,0}
\definecolor{MGreen}{rgb}{0,0.8,0}
\definecolor{DGreen}{rgb}{0,0.55,0}
\definecolor{Yellow}{rgb}{1,1,0}
\definecolor{Cyan}{rgb}{0,1,1}
\definecolor{Magenta}{rgb}{1,0,1}
\definecolor{Orange}{rgb}{1,.5,0}
\definecolor{Violet}{rgb}{.5,0,.5}
\definecolor{Purple}{rgb}{.75,0,.25}
\definecolor{Brown}{rgb}{.75,.5,.25}
\definecolor{Grey}{rgb}{.5,.5,.5}
\definecolor{Grey}{rgb}{.5,.5,.5}
\definecolor{Pink}{rgb}{1,0,1}
\definecolor{DBrown}{rgb}{.25,.17,.08}
\definecolor{DRed}{rgb}{0.5,0,0}
\definecolor{Black}{rgb}{0,0,0}

\newcommand{\PP}{\mathcal{P}}

\newcommand{\E}{\mathbb{E}}
\newcommand{\R}{\mathbb{R}}

\newcommand{\G}{\mathbb{G}}

\DeclareMathOperator*{\argmin}{argmin}
\DeclareMathOperator*{\argmax}{argmax}

\def\eps{\epsilon}

\newcommand{\defeq}{{\stackrel{def}{=}}}
\newcommand{\paren}[1]{{\left( {#1} \right)}}
\newcommand{\floor}[1]{{\left\lfloor {#1} \right\rfloor}}
\newcommand{\ceil}[1]{{\left\lceil {#1} \right\rceil}}
\newcommand{\bracket}[1]{{\left[ {#1} \right]}}
\newcommand{\braces}[1]{{\left\{ {#1} \right\}}}
\newcommand{\abs}[1]{{\left| {#1} \right|}}
\newcommand{\ip}[2]{{\left \langle {#1} , {#2}\right\rangle}}
\newcommand{\W}[1]{{W[{#1}]}}

\newcommand{\oracle}{\eps_k^{(O)}}
\newcommand{\horacle}{\hat \eps_k^{(O)}}

\newcommand{\approxerr}{{\nu}}
\newcommand{\Hn}{H_n}
\newcommand{\HnW}{\Hn(W)}
\newcommand{\Q}{Q}

\newcommand{\B}{{\mathcal{B}}}
\newcommand{\A}{{\mathcal{A}}}
\newcommand{\Lap}{{\text{\rm Lap}}}

\newcommand{\hscore}{{\widehat{\text{score}}}}
\newcommand{\hrho}{{\hat \rho}}
\newcommand{\maxH}{R}
\newcommand{\trho}{{\tilde \rho}}

\newtheorem{theorem}{Theorem}

\newtheorem{remark}{Remark}
\newtheorem{lemma}{Lemma}

\newtheorem{prop}{Proposition}
\newtheorem{coro}{Corollary}

\newtheorem{Defi}{Definition}
\newtheorem{defn}[Defi]{Definition}
\newtheorem{definition}[Defi]{Definition}

\newcommand{\drop}[1]{\relax}
\newcommand{\ignore}[1]{\relax}

\newcommand{\mypar}[1]{\paragraph{{#1}.}}

\author{Christian Borgs\thanks{Microsoft Research New England,
    Cambridge, MA, USA. Email: \url{{cborgs,jchayes}@microsoft.com}} \and
  Jennifer T. Chayes\footnotemark[1] \and Adam Smith\thanks{Pennsylvania
    State University, University Park, PA, USA. Email:
    \url{asmith@psu.edu}. Supported by NSF award IIS-1447700 and a Google Faculty Award. Part of this work was done
    while visiting Boston University's Hariri Institute for
    Computation and Harvard University's Center for Research on
    Computation and Society.}
}

\begin{document}

\title{
Private Graphon Estimation for Sparse Graphs}

\maketitle

\begin{abstract}
We design algorithms for fitting a high-dimensional statistical model
to a large, sparse network without revealing sensitive information of individual members.
Given a sparse input graph $G$, our
algorithms output a node-differentially-private nonparametric block model approximation.  By node-differentially-private, we mean that our output hides the insertion or removal of a vertex and all its adjacent edges.
If $G$ is an instance of the network obtained from a generative
nonparametric model defined in terms of a graphon $W$, our model
guarantees consistency, in the sense that as the number of
vertices
tends to infinity,
the output of our algorithm converges to $W$ in an appropriate version
of the $L_2$ norm.
In particular, this means we can estimate the sizes of all multi-way
cuts in $G$.

Our results hold as long as $W$ is bounded, the average degree of $G$
grows at least like the log of the number of vertices, and
the number of blocks goes to infinity
at an appropriate rate.
We give explicit error bounds in terms of the parameters of the
model; in several settings, our bounds improve on or match known
nonprivate results.
\end{abstract}

\newpage
{
\tableofcontents
}
\newpage

\section{Introduction}
\mypar{Differential Privacy}
Social and communication networks have been the subject of intense study over the
last few years.  However, while these networks comprise a rich source of
information for science, they also contain highly sensitive private information. What kinds of information can we release about these networks while preserving the
privacy of their users? Simple measures,
such as removing
obvious identifiers, do not
work; for example, several studies (e.g., \citep{BDK07,NS09})
reidentified individuals in the graph of a social network even after all vertex and
edge  attributes were removed.
Such attacks highlight the need for statistical and learning
algorithms that provide rigorous privacy guarantees.

 Differential privacy \citep{DMNS06},
which emerged from a line of work started by \cite{DiNi03},
provides meaningful guarantees in the
presence of arbitrary side information. In a traditional statistical
data set, where each person corresponds to a single record (or row of
a table), differential privacy guarantees that adding or removing any
particular person's data will not noticeably change the distribution
on the analysis outcome. There is now a rich and deep literature on
differentially private methodology for learning and other algorithmic
tasks; see \cite{DworkRoth14book} for a recent tutorial.
By contrast, differential privacy in the context of graph data is
 much less developed.  There are two main variants of  graph differential privacy:
\emph{edge} and \emph{node} differential privacy. Intuitively, edge differential privacy ensures that an algorithm's output does not reveal the inclusion or removal of a particular edge in the graph, while node differential privacy hides the inclusion or removal of a node together with all its adjacent edges.
Edge privacy is a weaker notion (hence
easier to achieve) and has been studied more extensively, with
particular emphasis on the release of individual graph statistics
\citep{NRS07,RHMS09,KRSY11,MirW12,LuM14exponential,KarwaSK14}, the
degree distribution \citep{HayLMJ09,HayRMS10,KarwaS12,LinKifer13,KarwaS15}, and
data structures for estimating the edge density of all cuts in a graph
\citep{GRU12,BlockiBDS12jl}.
Several authors designed edge-differentially private algorithms for
fitting generative graph models
\citep{MirW12,KarwaSK14,LuM14exponential,KarwaS15,XiaoCT14}, but these do
not appear to generalize to node privacy with meaningful
accuracy guarantees.

The stronger notion, node privacy, corresponds more closely to what
 was achieved in the case of traditional data sets, and to what one would
 want to protect an individual's data:  it
ensures that \emph{no matter what an analyst observing the released
information knows ahead of time}, she learns the
same things about an individual  Alice regardless of whether
Alice's data are used or not.
In particular, no assumptions are needed on the way the
individuals' data are generated (they need not even be
independent).   Node privacy was studied more recently
\citep{KNRS13,ChenZhou13,BBDS13, RaskhodnikovaS15}, with a focus on
on the release of descriptive statistics (such as the number of
triangles in a graph).
Unfortunately, differential privacy's stringency
makes the design of accurate, node-private algorithms challenging.

In this work, we provide the first algorithms for node-private
inference of a high-dimensional statistical model that does not admit
simple sufficient statistics.

\mypar{Modeling Large Graphs via Graphons}
Traditionally, large graphs have been modeled using various parametric models,
one of the most popular being the stochastic block  model
\citep{HLL83}.
Here one
postulates that an observed graph was generated by first assigning
vertices at random to one of $k$ groups, and then
connecting two vertices with a probability that depends on the
groups  the two vertices are members of.

As the number of vertices of the graph in question grows, we
  do not expect the graph
 to be well described by a stochastic block model with a fixed
number of blocks.   
In this paper we consider nonparametric models (where the number of
parameters need not be fixed or even finite) given in terms of a
\emph{graphon}. A graphon is a measurable, bounded function $W:[0,1]^2\to
[0,\infty)$ such that $W(x,y)=W(y,x)$, which for convenience we take
to be normalized: $\int W=1$.  Given a graphon, we
generate a graph on $n$ vertices by first assigning i.i.d. uniform
labels in $[0,1]$ to the vertices, and then connecting vertices with
labels $x,y$ with probability $\rho_n W(x,y)$, where $\rho_n$ is a
parameter determining the density of the generated graph $G_n$
with $\rho_n\|W\|_\infty\leq 1$.   We call $G_n$ a
$W$-random graph with target density $\rho_n$ (or simply a
$\rho_nW$-random graph).

To our knowledge, random graph models of the above form
 were first introduced
under the name latent
position graphs \citep{HRH02}, and are special cases of a more general
model of ``inhomogeneous random graphs'' defined in \cite{BJR07},
which is the first place were $n$-dependent target densities $\rho_n$
were considered.  For both dense graphs (whose target density does not
depend on the number of vertices) and sparse graphs (those for which
$\rho_n\to 0$ as $n\to\infty$), this model is related to the
theory of convergent graph sequences
\citep{BCLSV06,BCLSV08,BCLSV12,BCCZ14a,BCCZ14b}.
  For dense graphs it
was first explicitly proposed in \cite{LS06},   
though
it can be
implicitly traced back to \cite{H79,A81}, where models of this form appear as extremal points of
two-dimensional exchangeable arrays; see \cite{DJ08}
(roughly, their results relate graphons to exchangeable arrays the way de Finetti's theorem relates i.i.d. distributions to
exchangeable sequences).  
For sparse graphs,  \cite{CF15} offers a different nonparametric approach.

\mypar{Estimation and Identifiability}
Assuming that $G_n$ is generated in this way, we are then faced with
the task of estimating $W$ from a {\it single observation} of a graph $G_n$.
To our knowledge, this task was first explicitly considered in
\cite{BC09}, 
which considered graphons describing stochastic block models with a
fixed number of blocks.  This was generalized to models with a growing
number of blocks \cite{RCY11,CWA12},
while the first estimation of the nonparametric model was proposed in
\cite{BCL11}. 
Various other estimation methods were proposed recently, for example
\cite{LOGR12, 
 TSP13, 
 LR13, 
 WO13, 
 CA14,
 ACC13, YangHA14, 
 GaoLZ14, 
ABH14,
 Chatterjee15,
AS15known,AS15unknown}.
These works make various assumptions on the function $W$, the most common one
being that after a measure-preserving transformation, the integral of
$W$ over one variable is a strictly monotone function of the other,
corresponding to an asymptotically strictly monotone degree
distribution of $G_n$. (
This assumption is quite
restrictive: in particular, such results do not apply to graphons
that represent block models.)
For our purposes, the most relevant works are
\citet{WO13,GaoLZ14,Chatterjee15} and \citet{AS15unknown}, which provide consistent estimators
without
monotonicity assumptions
(see ``Comparison to
nonprivate bounds'', below).

{\setlength{\abovedisplayskip}{9pt}\setlength{\belowdisplayskip}{9pt}
One issue that makes estimation of graphons challenging is
\emph{identifiability}: multiple graphons can lead to the same
distribution on $G_n$.
Specifically, two graphons $W$ and $\tilde W$ lead to the
same distribution on $W$-random graphs if and only if there are measure preserving
maps $\phi,\tilde\phi:[0,1]\to[0,1]$ such that $W^\phi=\widetilde
W^{\widetilde \phi}$, where $W^\phi$ is defined by
$W(x,y)=W(\phi(x),\phi(y))$ \cite{DJ08,BCL10}.
Hence, there is no
``canonical graphon'' that an estimation procedure can output, but rather
an equivalence class of graphons.
Some of the literature circumvents identifiability by making strong
additional assumptions, such as strict monotonicity, that imply the
existence of  canonical
equivalent class representatives. We make no such assumptions,
but instead
define consistency in terms of a metric
on these equivalence classes, rather than on graphons as functions.
We 
use a
variant of the $L_2$ metric,
 \begin{equation}
   \label{delta-def}
\delta_2(W,W')=\inf_{\phi:[0,1]\to[0,1] }
 \|W^\phi-W'\|_2\, .
\end{equation}
where $\phi$ ranges over measure-preserving bijections.

}

\subsection{Our Contributions}

In this paper we construct an algorithm that produces
 an estimate $\hat W$ from a single instance
$G_n$ of a $W$-random graph with target density $\rho_n$
(or simply $\rho$, when $n$
is clear from the context).
We aim for several properties:
\begin{compactenum}
\item $\hat W$ is differentially private;
\item
$\hat W$ is consistent, in the sense that
$\delta_2(W,\hat W)\to 0$  in probability as $n\to\infty$;
\item $\hat W$ has a compact representation (in our case, as a matrix
  with $o(n)$ entries);
\item The procedure works for sparse graphs, that is, when the
  density $\rho$ is small;
\item On input $G_n$, $\hat W$ can be calculated efficiently.
\end{compactenum}

Here we give an estimation procedure that obeys the
first four properties, leaving the question of polynomial-time
algorithms for future
work.  Given an input graph $G_n$, a privacy-parameter $\eps$ and a
target number $k$ of blocks, our algorithm $\A$ produces a $k$-block graphon $\hat W=\A(G_n)$
such that
 \begin{itemize}
 \item $\A$ is $\eps$-differentially node private. The privacy
   guarantee holds for all inputs, independent of modeling assumptions.
 \item Assume that
   \begin{inparaenum}\renewcommand{\labelenumi}{(\arabic{enumi})}
   \item
     $W$ is an arbitrary graphon, normalized so $\int W =1$;
   \item the
     expected average degree $(n-1)\rho$ grows at least as fast as $\log
     n$; and
\item $k$ goes to infinity sufficiently slowly with $n$.
   \end{inparaenum}
   Then when $G_n$ is $\rho W$-random, the estimate $\hat W$ for $W$
   is \emph{consistent} (that is, $\delta_2(\hat W,W)\to 0$, both in
   probability and almost surely).
\end{itemize}
Combined with the general theory
of convergent graphs sequences, these result in particular give a node-private procedure
for estimating the edge density of all cuts in a $\rho W$-random graph,
see \eqref{cut-bound} in Section~\ref{sec:W-rand}
below.

The main idea of our algorithm is to use the exponential mechanism of
\cite{MT07} to select a block model which approximately minimizes the
$\ell_2$ distance to the observed adjacency matrix of $G$, under the
best possible assignment of nodes to blocks (this explicit search over
assignments makes the algorithm take exponential time). 
In
order to get an algorithm that is accurate on sparse graphs, we need several
nontrivial extensions of current techniques.
To achieve privacy, we use a new variation of the
Lipschitz extension technique of \cite{KNRS13,ChenZhou13} to reduce the
sensitivity of the $\delta_2$ distance. While those works used
Lipschitz extensions for noise addition, we use of Lipshitz extensions inside the
``exponential mechanism''~\cite{MT07} (to control the
sensitivity of the score functions). To bound our algorithm's error, we provide
a new analysis of the $\ell_2$-minimization algorithm; we show that approximate minimizers
are not too far from the actual minimizer (a ``stability'' property).
Both aspects of our work are enabled by restricting
the $\ell_2^2$-minimization to a set of block models whose density (in fact,
$L_\infty$ norm) is not much larger than that of the underlying graph.
The algorithm is presented
in Section~\ref{sec:estimation}.

Our most general result proves consistency for arbitrary graphons $W$ but does not
provides a concrete rate of convergence. However, we provide explicit rates under various assumptions on $W$.
Specifically, we relate the
error of our estimator to two natural error terms involving the graphon $W$:
the error $\oracle(W)$ of the best $k$-block
approximation to $W$ in the $L_2$ norm (see \eqref{eps-k} below) and  an error term $\eps_n(W)$
measuring the $L_2$-distance between the  graphon $W$
and the matrix of probabilities $\HnW$ generating the graph $G_n$
(see \eqref{eps-n} below.)
In terms of these error terms, Theorem~\ref{thm:final}
shows
\begin{equation}
\delta_2\Bigl(W,\hat W\Bigr)
  \leq \oracle(W) +  2 \eps_n(W) +
   O_P\paren{ \sqrt[4]{\frac{\log k}{\rho n}
}
 +\sqrt{\frac{k^2\log n}{n\eps}} + \frac{1}{\rho \eps n}}.\label{eq:main-bound}
\end{equation}

Along the way, we provide a novel analysis of a straightforward,
nonprivate least-squares estimator, whose error bound has a better
dependence on $k$:
\begin{equation}
\delta_2\Bigl(W,\hat W_{\text{nonprivate}}\Bigr)
  \leq \oracle(W) +  2 \eps_n(W) +
   O_P\paren{ \sqrt[4]{\frac{\log k}{\rho n}
       +\frac {k^2}{\rho n^2}
}
}. \label{eq:main-bound-nonprivate}
\end{equation}

It follows from the theory of graph convergence that
for all graphons $W$, we have
$\oracle(W)\to 0 $ as $k\to \infty$ and
$\eps_n(W)\to 0$ almost
surely as $n\to \infty$.
As proven in Appendix~\ref{sec:sampling-k-block}, we also have $\eps_n(W)
=O_P(\oracle(W)+\sqrt[4]{k/n})$, though this upper bound is loose in
many cases.

As a specific instantiation of these bounds, let us consider the
case that $W$ is exactly described by a $k$-block
model, in which case  $\oracle(W)=0$ and
$\eps_n(W)=O_P(\sqrt[4]{k/n})$ (see Appendix~\ref{sec:sampling-k-block}). For  $k\leq (n/\log^2 n)^{1/3}$, $\rho\geq\log(k)/k$
and constant $\eps$, our
private estimator has an asymptotic error that is dominated by the (unavoidable) error of
$\eps_n(W)=\sqrt[4]{k/n}$, showing that we do not lose anything due to privacy in
this special case.  Another special case is when $W$ is $\alpha$-H\"older continuous, in which case $\oracle(W)=O(k^{-\alpha})$ and $\eps_n(W)=O_P(n^{-\alpha/2})$;
see Remark~\ref{rem:Hoelder-Cont-W} below.

\mypar{Comparison to Previous Nonprivate Bounds}
We provide the first consistency bounds for estimation of a
nonparametric graph model subject to node differential
privacy. Along the way, for sparse graphs, we provide more general consistency results
than were previously known, regardless of privacy.
In particular, to the best of our knowledge, \emph{no prior results
give a consistent estimator for $W$ that works for sparse graphs without any additional assumptions besides
boundedness.}

When compared to results for nonprivate algorithms applied to
graphons obeying additional assumptions, our bounds are often
incomparable, and in other cases match the existing bounds.

We start by considering graphons which are themselves
step functions with a known number of steps $k$.  In the dense
case, the nonprivate algorithms of \cite{GaoLZ14}
and \cite{Chatterjee15}, as well as our nonprivate algorithm,
give an asymptotic error that is dominated by the term
$\eps_n(W)=O(\sqrt[4]{k/n})$, which is of the same order
as our private estimator as long as $k=\tilde o (n^{1/3})$.
\cite{WO13}  provided the first convergence results
for estimating graphons in the sparse regime. Assuming that $W$ is
bounded above and below (so it takes values in a range
$[\lambda_1,\lambda_2]$ where $\lambda_1>0$), they
analyze an inefficient algorithm (the MLE).
%
The bounds of \cite{WO13} are incomparable to ours, though for the
case of $k$-block graphons, both their bounds and our nonprivate bound are dominated
by the term $\sqrt[4]{k/n}$ when $\rho>(\log k) / k$ and $k \leq \rho
n$. A different sequence of works shows how to consistently estimate the underlying
block model with a \emph{fixed} number of blocks $k$
in polynomial time for
very sparse
graphs (as for our non-private algorithm, the only thing which is needed is
that 
$n\rho \to \infty$)
\cite{ABH14,AS15known,AS15unknown}; we are not aware of concrete
bounds on the convergence rate.

For the case of \emph{dense} $\alpha$-H\"older-continuous graphons, the results
of \cite{GaoLZ14} give an error  which is dominated
by the  term $\eps_n(W)=O_P(n^{-\alpha/2})$.  For $\alpha<1/2$,
our nonprivate bound matches this bound, while for $\alpha >1/2$ it is worse.
\cite{WO13} consider the sparse case.
The rate of
their estimator is incomparable to the rate of
our estimator; further, their analysis requires a lower bound on
the edge probabilities, while ours does not.

See Appendix~\ref{sec:comparisons} for a more detailed discussion of
the previous literature.

\section{Preliminaries}

\subsection{Notation}

For a
graph $G$ on $[n]=\{1,\dots,n\}$, we use $E(G)$ and $A(G)$ to denote the edge set and the
adjacency matrix of $G$, respectively.  The edge density $\rho(G)$ is defined as the number of edges divided by $n\choose 2$. Finally the degree $d_i$ of a vertex $i$ in $G$ is the number of edges containing $i$.
We use the same notation for a weighted graph with nonnegative edge weights $\beta_{ij}$, where now $\rho(G)=\frac 2{n(n-1)}\sum_{i<j}\beta_{ij}$,
and $d_i=\sum_{j\neq i}\beta_{ij}$.
We use   $\G_n$ to denote the set of weighted graphs on $n$ vertices with
weights in $[0,1]$, and  $\G_{n,d}$ to denote the set of all graphs
in $\G_n$ that have maximal degree at most $d$.

\mypar{From Matrices to Graphons}
We define a graphon to be a
bounded, measurable function $W:[0,1]^2\to \R_+$ such that $W(x,y)=W(y,x)$ for all $x,y\in [0,1]$.  It will be convenient to embed the set of
 a symmetric $n\times n$ matrix with nonnegative entries into graphons
as follows: let $\PP_n=(I_1,\dots I_n)$ be the partition of $[0,1]$
into
adjacent intervals of lengths $1/n$.
Define $\W{A}$ to be the step function which equals
$A_{ij}$ on $I_i\times I_j$.
If $A$ is the adjacency matrix of an unweighted graph $G$, we use $\W{G}$ for $\W{A}$.

\mypar{Distances}
For $p\in [1,\infty)$ we define the $L_p$ norm
of an $n\times n$ matrix $A$ by $\|A\|_p^p=\frac 1{n^2}\sum_{i,j}|A_{ij}|^p$,
  and the $L_p$ norm of a (Borel)-measurable
function $W:[0,1]^2\to \R$ by $\|f\|_p^p=\int |f(x,y)|^pdxdy$.  Associated with the $L_2$-norm is a scalar
product, defined as $\langle A,B \rangle=\frac 1{n^2}\sum_{i,j}A_{ij}B_{ij}$ for two $n\times n$ matrices $A$ and $B$, and
$\langle U,W\rangle=\int U(x,y)W(x,y) dxdy$ for two square integrable functions $U,W:[0,1]^2\to\R$.
Note that with this notation,
the edge density and the $L_1$ norm are related by $\|G\|_1=\frac{n-1}{n}\rho(G)$.

Recalling \eqref{delta-def}, we define the $\delta_2$ distance between two matrices $A,B$, or between a matrix $A$ and
a graphon $W$
by
$\delta_2(A,B)=\delta_2(\W{A},\W{B})$ and $\delta_2(A,W)=\delta_2(\W{A},W)$.  In addition, we will also use the in general larger distances $\hat\delta_2(A,B)$ and $\hat\delta_2(A,W)$, defined by taking a minimum over matrices $A'$ which are obtained from
$A$ by a relabelling of the indices:
$\hat\delta_2(A,B)=\min_{A'}\|A'-B\|_2$ and $\hat\delta_2(A,W)=\min_{A'}\|A'-W\|_2$.

\subsection{$W$-random graphs and graph convergence}
\label{sec:W-rand}

\mypar{W-random graphs and stochastic block models}
Given a graphon $W$ we define a random $n\times n$ matrix $\Hn=\HnW$ by
choosing $n$ ``positions'' $x_1,\dots, x_n$ i.i.d. uniformly at
random from $[0,1]$ and then setting $(\Hn)_{ij}=W(x_i,x_j)$.  If
$\|W\|_\infty\leq 1$, then $\HnW$ has entries in $[0,1]$, and we can form a
random graph $G_n=G_n(W)$ on $n$-vertices by choosing an edge between two vertices $i<j$ with probability $(\Hn)_{ij}$, independently for all $i<j$.
Following \cite{LS06} we call $G_n(W)$ a $W$-random graph and $\HnW$ a $W$-weighted random graph.
We incorporate a target density $\rho_n$ (or simply $\rho$, when $n$
is clear from the context) by normalizing $W$ so that $\int W=1$ and
taking $G$ to be a sample from $G_n(\rho W)$.  In other words,
we set $\Q=\Hn(\rho W)=\rho \HnW$ and then connect $i$ to $j$ with
probability $Q_{ij}$, independently for all $i<j$.

The error 
from our main estimates
 measuring the distance
between $\HnW$ and $W$  is defined as
\begin{equation}
\label{eps-n}
\eps_n(W)=\hat\delta_2(\HnW,W)
\end{equation}
and goes to zero as $n\to\infty$ by the following lemma,
which follows easily from the results of \cite{BCCZ14a}.
\begin{lemma}
\label{lem:\Hn-Convergence}
Let $W$ be a graphon with $\|W\|_\infty{<\infty}$.
With probability one,
$\|\HnW\|_1 \to \|W\|_1$ and
$\eps_n(W)\to 0.$
\end{lemma}
Stochastic block models are specific examples of $W$-random graph in
which $W$ is constant on sets of the form $I_i\times I_j$, where
$(I_1,\dots, I_k)$ is a partition of $[0,1]$ into intervals of
possibly different lengths.

{
\mypar{Approximation by block models}
In the opposite direction, we can map a function $W$ to a matrix $B$ by the following  procedure.  Starting from an
arbitrary partition $\PP=(Y_1,\dots,Y_k)$ of $[0,1]$ into sets of equal Lebesgues measure, define $W/\PP$ to be the
matrix obtained by averaging over sets of the form $Y_i\times Y_j$,
\[
(W/\PP)_{ij}=\frac 1{\lambda(Y_i)\lambda(Y_j)}
\Bigl(\int_{Y_i\times Y_j}W(x,y)dxdy\Bigr)
\]
where $\lambda(\cdot)$ denotes the Lebesgue measure. Finally, we will  use $W_\PP$ to denote the step function
\[
W_\PP=\sum_{i,j\in [k]}(W/\PP)_{ij}1_{Y_i}\times 1_{Y_j}.
\]
Using the above averaging procedure, }
it is easy to see that any graphon $W$ can
be well approximated by a block model.  Indeed, let
\begin{equation}\label{eps-k}
\oracle(W)=\min_{B}\|W-\W{B}\|_2
\end{equation}
where the minimum goes over all $k\times k$ matrices $B$.
{Given that we are minimizing the $L_2$-distance, the minimimizer
can easily be calculated, and is equal to $W_{\PP_k}$,
where $\PP_k$ is a partition of $[0,1]$ into adjacent intervals of
lengths $1/k$.  It then follows from the Lebesgue density theorem
(see, e.g., \cite{BCCZ14a} for details) that}
$\oracle(W)=\|W-W_{\PP_k}\|_2\to 0$ as $k\to\infty$.

We will take the above approximation as a benchmark for our approach, and consider it the
error an ``oracle'' could obtain (hence the superscript $O$).


\mypar{Convergence}
\label{sec:convergence}
{The theory of graph convergence was first
developed
 \cite{BCLSV06,BCLSV08,BCLSV12}, where it was formulated for dense graphs,
and then generalized to sparse graphs in \cite{BR,BCCZ14a,BCCZ14b}.
One of the notions of graph convergence considered in these papers is
the notion of convergence in metric.  The metric in question
is  similar to the metric $\delta_2$, but instead of the
  $L_2$-norm, one starts from the cut-norm $\|\cdot\|_\square$ first
  defined in \cite{FK99},
  \[
  \|W\|_\square=\sup_{S,T\subset [0,1]} \Bigl| \int_{S\times T}W
  \Bigr|,
  \]
  where the supremum goes over all measurable sets $S,T\subset [0,1]$.
  The cut-distance $\delta_\square$ between two integrable functions
  $U,W;[0,1]^2\to\R$ is then defined as
  \[
  \delta_\square(U,W)=\inf_{\phi}\|U^\phi-W\|_\square,
  \]
  where the inf goes over all measure preserving bijections on
  $[0,1]$.
 We will also need the following variations:  a distance  $\hat\delta_\square(G,G')$
  between two graphs on the same node set, as well as a distance
  $\hat \delta_\square(G,W)$ between a graph $G$ and  a graphon $W$,
  defined as
 $
  \hat d_\square(G,G')=\min_{G''}\|\W{G''}-\W{G}\|_\square
$
and
$
  \hat d_\square(G,W)=\min_{G''}\|\W{G''}-W\|_\square,
$ respectively,
  where the minimum goes over graphs $G''$ isomorphic to $G$.

  Given these notions, we say a (random or deterministic) sequence
  $G_n$ of graphs converges to a graphon $W$ in the cut metric if, as $n\to\infty$,
  \[
  \delta_\square\paren{\frac 1{\rho(G_n)}\W{G_n},W}\to 0\, .
  \]

With this notion of convergence, for any graphon $W$ with
  $\int W=1$, a sequence of $W$-random graphs
$G_n$ with target density $\rho_n$ converges to the generating graphon
$W$. This} was shown for
bounded $W$ and $n$-independent target densities $\rho$ with $\rho\|W\|_\infty\leq 1$
in
  \cite{BCLSV08}, but
  the statement is much more general, and in particular holds for
  arbitrary target densities $\rho_n$ as long as $n\rho_n\to\infty$
  and $\limsup\rho_n\|W\|_\infty\leq 1$ \cite{BCCZ14a}.


\mypar{Estimation of Multi-Way Cuts}
Using the results of \cite{BCCZ14b}, the convergence of $G_n$ in the cut-metric $\delta_\square$
implies many interesting results for estimating various  quantities
defined on the graph $G_n$.  Indeed, a consistent approximation $\hat W$ to $W$ in
 the metric $\delta_2$ is clearly consistent in the weaker metric $\delta_\square$.
But this distance controls various quantities of interest to
computer scientists, e.g., the size of all multi-way cuts, implying that
a consistent estimator for $W$ also gives consistent estimators for
all multi-way cuts.

To formalize this, we need some notation.  Given a weighted
graph $G$ on $[n]$ with node-weights one and edge-weights $\beta_{xy}(G)$, and
given a
partition $\mathcal P=(V_1,\dots,V_q)$ of $[n]$ into $q$ groups, let
 $G/\PP$  be the weighted graph
  with weights
  \[
  \alpha_i(G/\PP)=|V_i|/|V(G)|
  \quad\text{and}\quad
 \beta_{ij}(G/\PP)=\frac 1{n^2\|G\|_1}\sum_{x\in V_i,y\in V_j}\beta_{xy}(G).
  \]
We call $G/\PP$ a $q$-quotient or $q$-way cut of $G$,
  and denote the set of all $q$-way cuts by  $S_q(G)$:
  \[
  S_q(G)=\{G/\PP : \PP\ \text{is a partition of $[n]$ into $q$
    sets }\}\, .
  \]

  We also  consider the set of \emph{fractional
  $q$-way cuts}, $\hat S_q(G)=\{G/\rho\}$, defined in terms of \emph{fractional $q$-partitions} $\rho$.
  A fractional $q$-partition
  of $V(G)$ is a map $\rho:V(G)\to \Delta_q: x\mapsto\rho(x)$, where
  $\Delta_q$ is the simplex
  $\Delta_q=\{\rho=(\rho_i)\in [0,1]^q\colon \sum_i\rho_i=1\}$, and the
  corresponding fractional quotient $G/\rho$ is the weighted graph with weights
  $\alpha_i(G/\rho)=\frac 1{|V(G)|}\sum_x\rho_i(x)$
  and $\beta_{ij}(G/\rho)=\frac 1{n^2\|G\|_1}\sum_{x\in V_i,y\in V_j}\beta_{xy}(G)\rho_i(x)\rho_j(x)$.

  The set of fractional $q$-partitions of a graphon $W$, $\hat S_q(W)$, is defined similarly:  $
  \hat S_q(W)=\{ W/\rho\mid \rho:[0,1]\to\Delta_q\}
  $,
  with a fractional partition now a measurable function $\rho:[0,1]\to \Delta_q$, and $W/\rho$
  given in terms of the weights
  \[
  \alpha_i(W/\rho)=\int\rho_i(x)dx
  \quad\text{and}\quad
  \beta_{ij}(W/\rho)=\frac1{\|W\|_1}\int\rho_i(x)\rho_j(y) W(x,y).
  \]
  To measure the distance between the various sets of $q$-way cuts, we use the Hausdorff distance
 between  sets  $S,S'\subset\R^{q+q^2}$,
 \[d_\infty^{\text{Haus}}(S,S')=
 \max\{\sup_{H\in S}\inf_{H'\in S'}\|H-H'\|_\infty,
 \;\sup_{H'\in S'}\inf_{H\in S}\|H-H'\|_\infty\},
 \]
 where $\|\cdot\|_\infty$ is the $L_\infty$-norm
$\|H-H'\|_\infty=\max\{\max_i|\alpha_i(H)-\alpha_i(H')|
,\max_{ij}|\beta_{ij}(H)-\beta_{ij}(H')|\}$.

It was shown in \cite{BCCZ14a} that if
$G_n$ converges to  $W$ in the cut-metric, then
$d_\infty^{\text{Haus}}(S_q(G_n),\hat S_q(W))\to 0$. In particular, a consistent estimator
$\hat W$
for the generating graphon $W$ of a $W$-random graph $G_n$ leads to a
consistent estimator $\hat S_q(\hat W)$ for the
cuts $S_q(G_n)$, in the sense that
$d_\infty^{\text{Haus}}(S_q(G_n),\hat S_q(\hat W))\to 0$.  With a little
more work, we can give quantitative bounds; see Theorem~\ref{thm:cuts} below.


%

\subsection{Differential Privacy for Graphs}
\label{sec:diff-P}

The goal of this paper is the development of a differentially private
algorithm for graphon estimation. The privacy guarantees are
formulated for worst-case inputs --- we do not assume that $G$ is
generated from a graphon when analyzing privacy. This ensures that
the guarantee remains meaningful no matter what an analyst knows ahead
of time about $G$.

In this paper, we consider the notion of node privacy.  We call
two
 graphs $G$ and $G'$ {\em node neighbors} if one can be obtained from the other by removing one node and its adjacent edges.

\begin{defn}[$\eps$-node-privacy]\label{def:dif-privacy}
A randomized algorithm $\A$ is \emph{$\eps$-node-private} if for all events $S$ in the output space of $\A$, and  node neighbors $G,G'$,
$$\Pr[\A(G) \in S] \leq \exp(\eps)\times \Pr [\A(G') \in S]\,.$$
\end{defn}
We also need the notion of the \emph{node-sensitivity} of a function $f:\G_n\to \R$,
defined as maximum  $\max_{G,G'}|f(G)-f(g')|$, where the maximum goes over node-neighbors.
{This constant is often called the Lipshitz constant of $f$}

{Finally, we need a lemma concerning the extension of
 functions $f:\G_{n,d}\to\R$ to functions
      $\hat f :\G_n\to \R$.
 We say a function on adjacency matrices is \emph{nondecreasing} if
  adding an edge to the adjacency matrix does not increase the value
  of the function.

    \begin{lemma}[\cite{McShane34,KNRS13}]
    \label{lem:Lip-extension}
      For every function $f:\G_{n,d}\to\R$, there is an extension
      $\hat f :\G_n\to \R$ of $f$ with the same node-sensitivity as $f$. If $f$ is a nondecreasing linear function of the adjacency
      matrix, then we can select $\hat f$ to be nondecreasing and
      computable in polynomial-time and so that $\hat f(G)\leq f(G)$
      for all graphs $G\in \G_n$.
    \end{lemma}

The lemma is proved in Appendix~\ref{app:private}.
}

\section{Differentially Private Graphon Estimation}
\label{sec:estimation}

\subsection{Least-squares Estimation}

Given a graph as input generated by an unknown graphon $W$, our goal
is to recover a block-model approximation to $W$. The basic nonprivate algorithm
we emulate is least squares estimation, which outputs the $k\times k$ matrix $B$
which is closest  to the input adjacency matrix $A$ in the distance
\[
\hat\delta_2(B,A)=
\min_{\pi}\|B_\pi-A\|_2,
\]
where the minimum runs over all
equipartitions $\pi$ of $[n]$ into $k$ classes, i.e., over all
maps $\pi:[n]\to [k]$ such that all classes have size as close to $n/k$ as possible,
i.e., such that $||\pi^{-1}({i})|-n/k|< 1$ for all $i$, and $B_\pi$ is the
$n\times n$ block-matrix with entries $(B_\pi)_{xy}=B_{\pi(x)\pi(y)}$.
If
$A$ is the adjacency matrix of a graph $G$, we write
$\hat\delta_2(B,G)$ instead of $\hat\delta_2(B,A)$.
In the above notation, the basic algorithm we would want to emulate is then the algorithm which outputs the
least square fit
$
\hat B=\argmin_B \hat\delta_2(B,G)
$,
where the $\argmin$ runs over all symmetric $k\times k$ matrices $B$.

\subsection{Towards a Private Algorithm}
Our main idea to turn the least square algorithm into a private algorithm
is to use the so-called exponential mechanism of McSherry and Talwar \cite{MT07}.
Applied naively, we would therefore want to output a random $k\times k$  matrix $B$ according to the probability distribution
\[
\Pr(\hat B = B)\propto\exp\paren{- C \hat\delta_2^2(B,A)},
\]
with $C$ chosen small enough to guarantee differential privacy.
It is a standard fact from the theory of differential privacy, that
$C$ should be at most $\eps$ over twice the node-sensitivity of the ``score function'', here $\delta_2^2(B,\cdot)$.
But this value of $C$ turns out to be too small to produce an output
that is a good approximation to the least square estimator.
Indeed,
for a given matrix $B$ and equipartition $\pi$, the distance
$\|G-B_\pi\|_2^2$ can change by as much as $\frac 1 n$ when $G$ is replaced by
a node-neighbor,
regardless of
the magnitude of the entries of $B$.  To obtain differential privacy, we
then would need to choose $C\geq  n\eps/2$, which turns out to not
produce useful results when
the input graph $G$ is  sparse, since
small values of $C$ will lead to large errors relative to the least square estimator.

To address this, we first note that we can work with an equivalent score
that is much less sensitive. Given $B$ and $\pi$, we subtract off the
squared norm of $G$ to obtain the following:
\begin{eqnarray}
\label{score}
score(B,\pi; G) &=&\|G\|^2_2 -\|G-B_\pi\|_2^2 = 2\ip{G}{B_\pi} -
\|B_\pi\|^2 \text{, and}\\
score(B;G)&=& \max_\pi score(B,\pi;G),
\end{eqnarray}
where the $\max$ ranges over equipartitions $\pi:[n]\to[k]$.  For a fixed input graph $G$, maximizing the score
is the same as minimizing the distance, i.e.
$
\argmin_B \hat\delta_2(B,G) = \argmax_B score(B;G).
$
The sensitivity of the new score is then bounded  by $\frac 2 {n^2}
\cdot \|B\|_\infty$ times the maximum degree in $G$
(since $G$ only affects the score via the
inner product $ \ip{G}{B_\pi}$).
But this is still problematic since, a priori,  we have no control over
either the size of
$\|B\|_\infty$ or the maximal degree of $G$.

To keep the sensitivity low, we  make two modifications: first, we  only optimize over matrices $B$ whose entries are of order $\rho_n$ (in the end, we expect that a good estimator will have entries which are not much larger than  $\|\rho_n W\|_\infty$, which is of order $\rho_n$), and second we  restrict ourselves to graphs $G$ whose maximum degree is not much larger than one would expect
for graphs generated from a bounded graphon, namely a constant times
the average degree.  While the first restriction is something we can
just implement in our algorithm, unfortunately the second is something
we have no control over: We need to choose $C$ small enough to
guarantee privacy for all input graphs, and we have set out to
guarantee privacy in the worst case, which includes graphs with
maximal degree $n-1$.  Here, we employ an idea from \citep{BBDS13,KNRS13}: we first consider the restriction of $score(B,\pi;\cdot)$ to
$\G_{n,d_n}$ where $d_n$ will be chosen to be of the order of the average degree of $G$, and then extend it
back to all graphs while keeping the sensitivity low.

\subsection{Private Estimation Algorithm}

After these motivations, we are now ready to define our algorithm.
 It takes as input
the privacy parameter $\eps$, the graph $G$, a number $k$ of blocks, and a constant $\lambda\geq 1$
that will have to be chosen large enough to guarantee consistency of the algorithm.
It outputs a matrix $B$ from the set of matrices
$$ \textstyle
\B_{\mu}= \{B\in [0,\mu]^{k\times k}: \text{all entries
          }B_{i,j}\text{ are multiples of }\frac 1 n\}.
$$
{Inside our algorithm, we use an $\eps/2$-private algorithm to get
an estimate $\hat \rho$ for the edge density of $G$.  We do so
by setting $\hrho=\rho(G)+\Lap(4/n\eps)$, where
$\Lap(\kappa)$  is a Laplace random variable with density $h(z)=\frac{1}{2\kappa}e^{-|z|/\kappa}$.
The existence of the Lipschitz extension used in the algorithm
follows from Lemma~\ref{lem:Lip-extension}.
}

\begin{algorithm}\caption{Private Estimation Algorithm}
\label{alg:main-algo}
\KwIn{$\eps>0$, $\lambda\geq 1$, an integer $k$ and graph $G$ on $n$
  vertices.}
\KwOut{$k\times k$ block graphon (represented as a $k\times k$ matrix)
estimating $\rho W$}

\label{step:rho-approx}Compute  an $(\eps/2)$-node-private density approximation $\hrho=\rho(G)+\Lap(4/n\eps)$ \;

$d=\lambda \hrho n$ (the target maximum degree) \;
$\mu = \lambda \hrho$ (the target
  $L_\infty$ norm for the matrix $B$) \;

For each $B$ and $\pi$, let $\hscore(B,\pi;\cdot)$ denote a
          nondecreasing Lipschitz extension of $score(B,\pi;\cdot)$ from
          $\G_{n,d}$ to $\G_n$ such that for all matrices $A$,
          $\hscore(B,\pi;A)\leq score(B,\pi;A)$, and define
          \[\hscore(B;A) = \max_{\pi} \hscore(B,\pi;A)\]

\Return $\hat B$, sampled from the distribution
\[
\Pr(\hat B = B)\propto\exp\paren{\frac \eps {4\Delta} \hscore (B;A)},
\]
where $B$ ranges over matrices in $\B_{\mu}$
and
$\displaystyle \Delta=\frac{4d \mu}{n^2}=\frac{4\lambda^2\hrho^2}{n}$ \;
\end{algorithm}

\begin{lemma}\label{lem:privacy-main}
Algorithm~\ref{alg:main-algo} is $\eps$-node private.
\end{lemma}

  \begin{proof}
    By Lemma~\ref{lem:DMNS} from Appendix \ref{app:private},
    the estimate $\hat \rho$ is
    $\eps/2$-private, so we want to prove that the exponential
    mechanism itself is $\eps/2$-private as well.  In view of
    Lemma~\ref{lem:exp-mech} from Appendix \ref{app:private},
    all we need to show is that the the
    vertex sensitivity of $\hscore(B;\cdot)$ is at most $\Delta$.  To
    this end, we first bound the vertex sensitivity of
%
    the original score when restricted to graphs with degree $d$. Let
    $G,G'\in\G_{n,d}$ be node neighbors.  From \eqref{score}, we see
    that
    \[{score}(B,\pi;G)-{score}(B,\pi;G') =\frac 2{n^2}\sum_{x,y\in
      [n]} (A_{xy}-A'_{xy})B_{\pi(x)\pi(y)}\, ,
    \]
    where $A,A'$ are the adjacency matrices of $G$ and $G'$. Since $A$
    and $A'$ differ in at most $2d$ entries, the score differs by at
    most $4d \|B_\pi\|_\infty/n^2$. This is at most $\Delta$, since
    $B\in \B_\mu$.  Since $\hscore$ is a Lipschitz extension of
    $score$, the vertex sensitivity of $\hscore$ (over \emph{all}
    neighboring graphs) is at most $\Delta$, as required.
  \end{proof}

\begin{theorem}[Performance of the Private Algorithm]\label{thm:final}
  Let $W:[0,1]^2\to [0,\Lambda]$ be a normalized graphon, let $0<\rho\Lambda\leq 1$, let $G=G_n(\rho W)$,
  $\lambda\geq 1$, and $k$ be an integer.  Assume that
  $\rho n\geq 6\log n$ and
  $8\Lambda\leq \lambda \leq \sqrt n$,
  ${2\leq k\leq\min\{n\sqrt{\frac\rho2} ,e^{\frac{\rho n}2}\} }$.
  Then the Algorithm~\ref{alg:main-algo} outputs an
  approximation $(\hat\rho,\hat B)$ such that
\[
\begin{aligned}
\delta_2\Bigl(W,\frac 1{\hat\rho}\W{\hat B}\Bigr)
  &\leq \oracle(W) +
 2\eps_n(W)
 +  O_P\paren{\sqrt[4]{\frac{\lambda^2\log k}{\rho n}
}
 +\lambda\sqrt{\frac{k^2\log n}{n\eps}}+\frac {{ \lambda}}{n\rho\eps}}.
  \end{aligned}
\]
\end{theorem}

The theorem will be proven in Section~\ref{sec:private-analysis}.

In the course of the proof,
we will prove results
on the performance of a non-private algorithm, which is a variant
of the standard least square algorithm, the main difference being that instead
of minimizing $\hat\delta_2(B,A)$ over all matrices  $B$, we only optimize it over
matrices whose entries are bounded by a constant times the density of $G$.

  \begin{algorithm}
    \caption{Nonprivate Algorithm}
    \label{alg:main-algo'}
    \KwIn{$\lambda\geq 1$, an integer $k$ and graph $G$ on $n$
      vertices.}  \KwOut{$k\times k$ block graphon (represented as a
      $k\times k$ matrix $B$) estimating $\rho W$}

    $\mu \gets \lambda \rho(G)$ (the target $L_\infty$ norm for the
    matrix $B$) \;

    \Return $\displaystyle \hat B \in {\argmin_{B\in\B_\mu} \hat \delta_2(B;G)}$.
  \end{algorithm}

\begin{theorem}[Performance of the Nonprivate Algorithm]\label{thm:final'}
  Let $W:[0,1]^2\to [0,\Lambda]$ be a normalized graphon, let
  $0<\rho\Lambda\leq 1$, let $G=G_n(\rho W)$, $\lambda\geq 1$, and $k$
  be an integer.  If $\hat B$ is the least-squares estimator
  (Algorithm~\ref{alg:main-algo'}), $2\Lambda\leq \lambda \leq \sqrt
  n$, $2\leq k\leq\min\{n\sqrt{\frac\rho2} ,e^{\frac{\rho n}2}\} $,
  then
\[
  \delta_2\Bigl(W,\frac 1{\rho(G)}\W{\hat B}\Bigr)
  \leq \oracle(W)+ 2\eps_n(W)
  +  O_p\paren{\sqrt[4]{\lambda^2\paren{\frac{\log k}{\rho n}
  +\frac {k^2}{\rho n^2}}}}.
  \]
In particular,
$
\delta_2\Bigl(W,\frac 1{\rho(G)}\W{\hat B}\Bigr)\to 0
$
in probability if
$k\to\infty$
and $\lambda^2\Bigl({\frac{\log k}{\rho n}
  +\frac {k^2}{\rho n^2}}\Bigr)\to 0$.
\end{theorem}
Theorem~\ref{thm:final'} is proven in Section~\ref{sec:Least-2-error}.

\begin{remark}
While
Theorem \ref{thm:final} and Theorem~\ref{thm:final'}
are
stated in term of bounds which hold in probability, our
proofs give slightly more, and allow us in particular to prove
statements which hold almost surely as $n\to\infty$.
Namely, they show that under the assumptions of Theorem~\ref{thm:final'}, the output $\hat B$ of
the nonprivate algorithm is such that
\[
\begin{aligned}
\delta_2\Bigl(W,\frac 1{\rho(G)}\W{\hat B}\Bigr)
  &\leq \oracle(W)+O\paren{\sqrt[4]{\frac{\lambda^2\log k}{\rho n}+\frac {\lambda^2k^2}{\rho n^2}}}+o(1);
  \end{aligned}
\]
they also show that if we replace the assumption $n\rho\geq 6\log n$ in Theorem~\ref{thm:final} by the
stronger assumption
$n\rho\eps/\log n\to\infty$, then the output $\hat B$ of
the private algorithm is such that
\[
\begin{aligned}
\delta_2\Bigl(W,\frac 1{\hat\rho}\W{\hat B}\Bigr)
 \leq \oracle(W)
 +  O\paren{\sqrt[4]{\frac{\lambda^2\log k}{\rho n}
}
 +\lambda\sqrt{\frac{k^2\log n}{n\eps}}
 {+\frac {\sqrt\lambda}{n\rho\eps}}
 }+o(1)
  \end{aligned}
\]
where in both expressions, $o(1)$ is a term which goes to zero with probability one as $n\to\infty$.

Thus for both algorithm, as long as $k$ grows sufficiently slowly with $n$, with probability one, the
asymptotic error is of the form $\oracle(W)+o(1)$, which is best possible, since we can't do better than
the best oracle block model approximation.
\end{remark}

\begin{remark}
\label{rem:Hoelder-Cont-W}
Under additional assumptions on the graphon $W$, we can say a little more.  For example, if we assume that
$W$ is H\"older continuous, i.e, if we assume that the exists constants $\alpha\in(0,1]$ and $C<\infty$
such that
$
|W(x,y)-W(x',y')|\leq C\delta^\alpha
$
whenever $|x-x'|+|y-y'|\leq\delta$, then we
have that $\oracle(W)=O(k^{-\alpha})$ and $\eps_n(W)=O_P(n^{-\alpha/2})$.
See Appendix~\ref{sec:Holder} for details.
\end{remark}

Theorems~\ref{thm:final} and ~\ref{thm:final'} imply that the sets of fractional $q$-way cuts of the estimator $\hat B$ from these theorems provide
good approximations to the $q$-way cuts of the graph $G$ (as defined
in Section~\ref{sec:W-rand}).  Specifically:

\begin{theorem}\label{thm:cuts}
Let $q\geq 2$ be an integer.

(i)
Under the assumptions of Theorem~\ref{thm:final'},
\[
d_\infty^{\text{Haus}}(S_q(G),\hat S_q(\hat B_{\text{nonprivate}}))=
O_p\paren{
\oracle(W)+ \eps_n(W)
  +  \sqrt[4]{\lambda^2\paren{{\frac{\log k}{\rho n}
  +\frac {k^2}{\rho n^2}}}}}.
\]

(ii) Under the assumptions of Theorem~\ref{thm:final},
\[
d_\infty^{\text{Haus}}(S_q(G),\hat S_q(\hat B_{\text{private}}))
=O_p\paren{
\oracle(W)+ \eps_n(W)+\sqrt[4]{\frac{\lambda^2\log k}{\rho n}
}
 +\lambda\sqrt{\frac{k^2\log n}{n\eps}}+\frac {{ \lambda}}{n\rho\eps}}.
\]
\end{theorem}
The proof of the theorem relies on the theory of graph convergence, in particular
 the results of
 \cite{BCLSV12,BCCZ14a,BCCZ14b}, and is given in Appendix~\ref{app:cuts}.
\begin{remark}
When considering the ``best'' block model approximation to $W$, one might want to
consider block models with unequal block sizes; in a similar way, one might want
to construct a private algorithm that outputs a block model with unequal size blocks,
and produces a bound in terms of this best block model approximation
instead of $\oracle(W)$.  With more cumbersome notation, this
can be easily proved with our methods,
with the minimal block size taking the role of $1/k$ in all our proofs. We leave the details to a journal version.
\end{remark}

\section{Estimation Error of the Least Square Algorithm}
\label{sec:Least-2-error}

At a high level, our proofs of Theorems~\ref{thm:final} and
of \ref{thm:final'}
follow from the fact that for all $B$ and $\pi$,
the expected score $\E[Score(B,\pi;G)]$
is equal to the score $Score(B,\pi;\Q)$, combined with a concentration
argument.  As a consequence, the maximizer $\hat B$
of $Score(B;G)$ will approximately minimize the $L_2$-distance $\hat\delta_2(B,\Q)$, which in turn will approximately
minimize $\|\frac 1\rho\W{B}- W\|_2$, thus relating the
$L_2$-error of our estimator $\hat B$ to the
``oracle error'' $\oracle(W)$ defined in \eqref{eps-k}.

In this section we present the analysis of exact and
  approximate least squares. This allows us to analyze the nonprivate
  algorithm. The analysis of the private algorithm (Theorem
\ref{thm:final}) requires additional arguments relating the private
approximate maximizer to the nonprivate one; we present these in
Section~\ref{sec:private-analysis}).

Our main concentration statement is contained in the following
proposition, which we prove  {in Section~\ref{sec:concentration} below.}
To state it, we define, for every symmetric $n\times n$ matrix $\Q$ with vanishing  diagonal, $Bern_{0}(\Q)$
to be the distribution over symmetric matrices $A$ with zero diagonal such that the entries $\{A_{ij}\colon i<j\}$ are independent Bernouilli random variables with $\E A_{ij}=\Q_{ij}$.

\begin{prop}\label{Cor:Concentration}
Let $\mu>0$, $\Q\in [0, {1}]^{n\times n}$  be a symmetric matrix with vanishing diagonal,
and $A \sim Bern_{0}(\Q)$.
If
    $2\leq k\leq\min\{n\sqrt{\rho(\Q) ,e^{\rho(\Q)n}\} }$
and $\hat B\in\B_\mu$ is such that
\[Score(\hat B;A)\geq \max_{B\in \B_\mu}Score(B;A)-\approxerr^2
\]
for some $\approxerr>0$,
then with probability at least $1-2e^{-n}$,
\begin{equation}
\label{hatdelta-bd}
\hat\delta_2(\hat B,\Q)\leq \min_{B\in \B_\mu}\hat\delta_2(B,\Q)+\approxerr +
O\paren{\sqrt[4]{\mu^2 \rho(\Q)\paren{\frac{k^2}{ n^2} + \frac {\log k}{ n}}}}
\end{equation}
and in particular 
\begin{equation}\label{hatb-2-bd}
\begin{aligned}
\|\hat B\|_2
&\leq (2\|\Q\|_2+ \approxerr)\paren{1+\frac {2k}n}
+ O\paren{\sqrt[4]{\mu^2 \rho(\Q)\paren{\frac{k^2}{ n^2} + \frac {\log k}{ n}}}}
\\
&\leq (2\|\Q\|_2+ \approxerr)\paren{1+\frac {2k}n}+ O\paren{\sqrt{\mu \rho(\Q)}}
\end{aligned}
\end{equation}
\end{prop}

Morally, the proposition contains almost all that is needed to
establish the bound  \eqref{eq:main-bound-nonprivate} proving consistency
of the standard least squares algorithm (which, in fact, only
involves the
case $\approxerr=0$), even though there are several additional steps needed to
complete the proof (see Sections~\ref{app:prop-proof} and
\ref{app:prop-proof-completed} below).


The proposition also contains an extra ingredient which is a
crucial input for the
analysis of the private algorithm: it states that if instead
of an optimal, least square estimator, we output an estimator whose
score is only approximately maximal, then the excess error introduced
by the approximation is small.  To apply the proposition, we then establish a
a lemma
which gives us a lower bound on the score of the output
$\hat B$ in terms of the maximal score and an excess error
$\approxerr$.

There are several steps needed to execute
this strategy, the most important ones involving
a rigorous control of the error introduced
by the Lipschitz extension inside the exponential algorithm
(which in turn requires estimating the deviation
of the maximal degree from the expected degree, a step
where the condition that $\rho n$ has to grow like $\log n$ is needed).
The excess error $\approxerr$ eventually turns into the second to last error term
in \eqref{eq:main-bound}, while the difference between
 our private estimator $\hat \rho$ for the edge density and the actual edge density
of $G$ is responsible for the last one.

The analysis of the private algorithm is presented in
Section~\ref{sec:private-analysis}; the remainder of this section
presents the detailed analysis of the least squares estimator.

\begin{remark}\label{rem:H-bound}
Note that for both the non-private algorithm and the private algorithm,
the above proposition naturally gives a bound for the $L_2$
estimation error for matrix of probabilities $\Q$.  In fact,
our proofs provide error bounds on $\hat\delta_2(\hat B,\Q)$
which differ from \eqref{eq:main-bound-nonprivate}, \eqref{eq:main-bound}
and the bounds in Theorem~\ref{thm:cuts}
in that
(i) the error term $2\eps_n(W)$ is absent, and
(ii) the oracle
error $\oracle(W)$ is replaced by an oracle error
$\horacle(\Hn)$ for $\Hn$,
see Theorems~\ref{thm:H'},  \ref{thm:H} and \ref{thm:cuts-from-H}.
Converting these bounds into bounds
on $\delta_2(W,\frac 1{\hat\rho}\W{\hat B})$ and expressing
the result in terms of
 $\oracle(W)$ instead of $\horacle(\Hn)$ then introduces the error term $2\eps_n(W)$
 in \eqref{eq:main-bound-nonprivate}, \eqref{eq:main-bound} and the bounds
 in Theorem~\ref{thm:cuts}.
\end{remark}

\subsection{Expectation and Concentration of Scores}
\label{sec:concentration}

The following two lemmas contain the core of the argument outlined at the beginning of
this section.
\begin{lemma}[Expected scores]\label{lem:expectations}
Let $\Q\in [0,1]^{n\times n}$ be a symmetric matrix with vanishing diagonal,  let $A \sim Bern_{0}(\Q)$,
and let $B,B'$ be $k\times k$ matrices.  Then
    \[
    \hat\delta_2^2(\Q,B)-\hat\delta_2^2(\Q,B')
    =\max_{\pi'}\E[Score(B',\pi';G)]-\max_{\pi}\E[Score( B,\pi;G)],
    \]
where the two max's go over equipartitions $\pi,\pi':[n]\to [k]$.
\end{lemma}

  \begin{proof}
By linearity of expectation, we have
    \begin{align*}
      \E score(B,\pi;A) &= \E\paren{2\ip{A}{B_\pi} - \|B_\pi\|^2 } =
      2\ip{\Q}{B_\pi} -\|B_\pi\|_2^2 \\
      &= \|\Q\|^2_2-\|\Q-B_\pi\|_2^2\,.
    \end{align*}
Taking into account the definition of $\hat\delta_2(B,\Q)$, the lemma follows.
  \end{proof}

Our second lemma states that the realized scores are close to their
expected values. The proof is based on a careful application of the
concentration bounds. The argument is delicate because we must take
advantage of the low density (when $\rho$ is small).

%
%
  \begin{lemma}[Concentration of scores]\label{lem:concentration}
    Let $\mu>0$, let $\Q\in [0, 1]^{n\times
      n}$  be a symmetric matrix with vanishing diagonal and let
      $A \sim Bern_{0}(\Q)$.
     If $2\leq k\leq\min\{n\sqrt{\rho(\Q)} ,e^{\rho(\Q)n}\} $, then, with probability at least
    $1-2e^{-n}$
\[
     |score(B,\pi;A) - \E\bracket{score(B,\pi;A)}| =
      O\paren{\mu \,\sqrt{\rho(\Q)\paren{\frac{k^2}{ n^2} + \frac {\log k}{ n}}}}
\]
for all equipartitions $\pi$ and all $B \in [0,\mu]^{k\times k}$.
  \end{lemma}

\begin{proof}
    First, consider a specific pair $B,\pi$.  Recall that
    $$score(B,\pi;A) - \E\bracket{score(B,\pi;A)} = 2\ip{A-\Q}{B_\pi}\,.$$
    We wish to bound the deviation of $score(B,\pi;A)$ from its
    mean. Set $\rho(\Q)=\trho$. The quantity $ S=\frac {n^2}{2\mu}\cdot
    \ip{A}{B_\pi}=\sum_{i<j}\frac{B_{\pi(i)\pi(j)}}\mu A_{ij}$ is a
    sum of $n\choose 2$ independent random variables in $[0,1]$ with
    expectation $\E{S}\leq \trho {n\choose 2}$.  Using a slight
    variation on the standard Chernoff bound, which we state in
    Lemma~\ref{lem:chernoff-mult}, we will bound the probability that
    $S$ deviates from its mean by at most $\beta\mu_0$, where
    $\mu_0\geq \E[S]$ will be chosen in a moment.  Setting
    $\eta=2e^{-n}$ and
    \[
    \beta=\sqrt{\frac{k^2 + n\log k+\log(2/\eta)}{3\trho n^2}}
    =O\paren{\sqrt{\frac{k^2}{\trho n^2} + \frac {\log k}{\trho n}}}
    \]
    the assumption $ k\leq\min\{n\sqrt{\rho(\Q)} n,e^{\rho(\Q)n}\} $
    implies $\beta\leq 1$, and setting $\mu_0=9n^2\trho$, the bound
    from Lemma~\ref{lem:chernoff-mult} becomes
    \[
    2e^{-{3\beta^2}\trho n^2} =e^{-k^2}k^{-n}\eta\leq
    2^{-k^2}k^{-n}\eta,
    \]
    implying that
      $$\Pr\left( |2\ip{A-\Q}{B_\pi}| \geq \frac{4\mu}{n^2}
        \beta\mu_0\right) \leq \frac\eta{k^n2^{k^2}}\,.$$

      Finally, we observe that for any $A$, the maximum of
      $|\ip{\Q-A}{B_\pi}|$ over all $B \in [0,\mu]^{k\times k}$ is the
      same as the maximum over all $B\in\{0,\mu\}^{k\times k}$.
      Taking a union bound over the (at most $2^{k^2}k^n$) pairs
      $B,\pi$ and observing that $ \frac{4\mu}{n^2}\beta\mu_0
      =O\paren{\mu \,\sqrt{\trho\paren{\frac{k^2}{ n^2} + \frac {\log
              k}{ n}}}} $, we get the statement of the lemma.
    \end{proof}

\begin{proof}[Proof of Proposition~\ref{Cor:Concentration}]
Let $\hat B\in\B_\mu$ be as specified, let $B'\in \B_\mu$ arbitrary, and  let
$\pi,\pi':{n}\to k$ be two equipartitions.
By Lemmas~\ref{lem:expectations} and \ref{lem:concentration},
\[
\begin{aligned}
\hat\delta_2^2(\Q,\hat B)-\hat\delta_2^2(\Q,B')
&=\max_{\pi'}\E[Score(B',\pi';G)]-\max_{\pi}\E[Score(\hat B,\pi;G)]
\\
&\leq Score(B';G)-Score(\hat B;G)+
O\paren{\mu \,\sqrt{\rho(\Q)\paren{\frac{k^2}{ n^2} + \frac {\log k}{ n}}}}
\\
&\leq \approxerr^2 + O\paren{\sqrt{\mu^2 \rho(\Q)\paren{\frac{k^2}{ n^2} + \frac {\log k}{ n}}}}
\end{aligned}
\]
which implies the bound \eqref{hatdelta-bd}.  (Taking square roots works since $\sqrt {\sum_i C_i^2}
\leq \sum_i C_i$ as long as $C_i\geq 0$.)
To prove \eqref{hatb-2-bd},
we use that for an arbitrary equipartition $\pi$
$\|\hat B_\pi\|_2^2\geq \Bigl(1-\frac kn\Bigr)\|\hat B\|_2^2$ and that
$\|\hat B_\pi\|_2\leq \|\Q\|_2+\|B_\pi -\Q\|_2$.  Inserting the definition
of $\hat\delta_2(\hat B,\Q)$ and using the main statement plus the assumptions $2\leq k\leq\min\{n\sqrt{\rho(\Q)} ,e^{\rho(\Q)n}\} $, we obtain \eqref{hatb-2-bd}.
\end{proof}

\subsection{Estimation of the edge-probability matrix $\Q$}
\label{app:prop-proof}

Up to technical details, Proposition~\ref{Cor:Concentration} contains all that is needed to prove consistency of the least square algorithm.
As indicated in Remark~\ref{rem:H-bound},
we will first prove that the algorithm gives a consistent estimator for the matrix $\Q$, and then use this prove that the output also gives a consistent estimator for $W$.  The first statement is formalized in the following theorem.

\begin{theorem}\label{thm:H'}
  Under the assumptions of Theorem~\ref{thm:final'},
\begin{equation}
\label{H'-bd}
  \hat\delta_2\Bigl(\frac 1{\rho(G)}\hat B,\HnW\Bigr)
  \leq \horacle(\HnW)
  +  O_p\paren{\sqrt\lambda\Bigl({\frac{\log k}{\rho n}
  +\frac {k^2}{\rho n^2}}\Bigr)^{1/4}}
  \end{equation}
  where
  \[\horacle(H)=
  \inf_B\hat\delta_2(B,H),
  \]
  with the $\inf$ going over all symmetric $k\times k$ matrices $B$.
  Moreover, a.s. as $n\to\infty$,
\[
\begin{aligned}
\hat\delta_2\paren{\frac 1{\rho(G)} \hat B,\HnW}
&\leq\horacle(\HnW) +
O\paren{\sqrt[4]{
        \lambda^2\paren{
           \frac{k^2}{ n^2\rho}
            + \frac{\log k}{ n\rho}
        }}}+o(1).
 \end{aligned}
\]
\end{theorem}

\begin{proof}
As a first step,
we will bound the left hand side of \eqref{H'-bd} by conditioning
on the event that
\begin{equation}
\label{good-H}
\frac\rho2\leq\rho(\Q)\leq 2\rho.
\end{equation}
By a concentration argument very similar to the proof of
Lemma~\ref{lem:concentration} above (in fact, it is easier, see
Lemma~\ref{lem:good-trho} (part 3) in Appendix~\ref{App:Aux}), we have that,
with probability at least $1-2e^{-n}$,
\[
\frac{\rho(G)}{\rho(\Q)}
=1+O\paren{\frac 1{\sqrt{n\rho(\Q)}}}=1+O\paren{\frac 1{\sqrt{n\rho}}}.
\]
We can now apply Proposition~\ref{Cor:Concentration} with $\approxerr
=0$ (since the nonprivate algorithm returns an exact
minimizer). Recall that $\HnW=\frac{\Q}{\rho}$ and $\mu
= \lambda \rho(G) = \Theta(\lambda \rho)$.
 We get that, with probability at least $1-4e^{-n}$,
\begin{equation}\label{apply-concentration}
\hat\delta_2\paren{\frac 1\rho \hat B,\HnW}
\leq\min_{B\in \B_\mu}\hat\delta_2\paren{\frac 1{\rho}B,\HnW} +
O\paren{
    \sqrt[4]{ { \lambda^2}
        \paren{
            \frac{k^2}{ n^2\rho}
            + \frac{\log k}{ n\rho}
        }
}}.
\end{equation}

In the remainder of the proof, we bound the first term on the
left-hand side above by relating it to the ``oracle error'' $\horacle(\HnW)$.
Let $B'$ and $\pi$ be such that $\horacle(\HnW)=\|\HnW-B'_\pi\|_2$.
It is easy to see that then  $B'_\pi$ is obtained from $\HnW$ by
averaging over the classes of  $\pi$, which in turn implies that
$
 \|B'_\pi\|_\infty
 =\|B'\|_\infty\leq \|\HnW\|_\infty
 \leq \|W\|_\infty\leq \Lambda
$
and $\|B'_\pi\|_2\leq \|\HnW\|_2=\rho^{-1}\|\Q\|_2$.
Define $B$ by rounding all entries of $\rho(G) B'$ down to the nearest
multiple of $1/n$, adding a rounding error of at most $1/n$,
so that $\|B-\rho(G) B'\|_\infty\leq 1/n$.
 Note that $B'$ is on the scale of $W$ and $\HnW$
(that is, we expect $\|B'\|=\Theta(1)$),
while $B$ is on the scale of $\rho W$ and $\Q$;
hence, $\|B_\pi\|_2\leq\frac{\rho(G)}\rho\|\Q\|_2$.
Now $\mu\geq \rho(G)\Lambda$, $\|B\|_\infty\leq
\rho(G)\|B'\|_\infty\leq \rho(G)\Lambda$, and $\Lambda \leq \lambda
{ /2}$.
Thus, $B$ is in the set $ \B_\mu$ that the algorithm searches
over. We can bound the first term in the left-hand side of
\eqref{apply-concentration} by
\begin{eqnarray*}
\hat\delta_2\paren{ \HnW,\frac 1{\rho}B}
&\leq & \Bigl\| \HnW-\frac 1{\rho}B_\pi\Bigr\|_2 \\
& \leq & \Bigl\| \HnW-\frac 1{\rho(G)}B_\pi\Bigr\|_2  +  \Bigl\| \frac
1{\rho(G)}B_\pi-\frac 1{\rho}B_\pi\Bigr\|_2  \\
& \leq & \horacle(\HnW) + \Bigl\| B'_\pi -\frac 1{\rho(G)}B_\pi\Bigr\|_2  +  \Bigl\| \frac
1{\rho(G)}B_\pi-\frac 1{\rho}B_\pi\Bigr\|_2  \\
& \leq & \horacle(\HnW) + \frac 1{n\rho(G)}{ +\Bigl|1 -}\frac{\rho(G)}\rho\Bigr|\frac{\|\Q\|_2}\rho
 \end{eqnarray*}
Combined with our previous two bounds
and the fact that
by \eqref{good-H}, we can bound $\|Q\|_2$ by
$
\|\Q\|_2\leq\sqrt{\|\Q\|_1\|\Q\|_\infty}
{ \leq \sqrt{\rho(Q)\Lambda\rho}
\leq\sqrt{2\Lambda}{\rho}
\leq \sqrt\lambda\rho
}
$,
this implies that
\[
 \begin{aligned}
\hat\delta_2\paren{\frac 1\rho \hat B,\HnW}
&\leq\horacle(\HnW)+
\sqrt{\lambda}\Bigl|1-\frac{\rho(\Q)}\rho\Bigr|
+O\paren{\sqrt{\frac{\lambda}{n\rho}}}
+O\paren{
    \sqrt[4]{\lambda^2
        \paren{
            \frac{k^2}{ n^2\rho}
            + \frac{\log k}{ n\rho}
        }        }}
\\
&\leq \horacle(\HnW)+
\sqrt{\lambda}\Bigl|1-\frac{\rho(\Q)}\rho\Bigr|+
O\paren{
    \sqrt[4]{
        \lambda^2\paren{
            \frac{k^2}{ n^2\rho}
            + \frac{\log k}{ n\rho}
        }
        }}.
 \end{aligned}
\]
We can now use  \eqref{hatb-2-bd} to bound
$\|\hat B\|_2$ by $O(\|\Q\|_2 + \sqrt{\mu \rho(\Q)}) = O(\sqrt\lambda \rho)$ and thus
$\delta_2(\hat B/\rho,\hat B/\rho(G))$
by  $O(\sqrt{\lambda})|1-{\rho(\Q)}/\rho|$
plus an error which can be absorbed into the error term above.
We obtain that, conditioned on \eqref{good-H}, with probability at least $1-4e^{-n}$, we have
\begin{equation}
\label{high-concentration}
 \begin{aligned}
\hat\delta_2\paren{\frac 1{\rho(G)} \hat B,\HnW}
&\leq\horacle(\HnW) +
O\paren{\sqrt[4]{
        \lambda^2\paren{
            \Bigl|1-\frac{\rho(\Q)}\rho\Bigr|^4
            +\frac{k^2}{ n^2\rho}
            + \frac{\log k}{ n\rho}
        }}}.
 \end{aligned}
\end{equation}
By Lemma~\ref{lem:good-trho} from Appendix~\ref{App:Aux},
$|{\rho(\Q)}-\rho|^2=O_P(\lambda \rho^2/n)$, implying that
\[
\lambda^2\Bigl|1-\frac{\rho(\Q)}\rho\Bigr|^4=
O_P\paren{\frac{\lambda^4}{n^2}}=O_P\paren{\frac{\lambda^2}{ n}}
=O_P\paren{\frac{\lambda^2\log k}{\rho n}}.
\]
On the other hand, again by Lemma~\ref{lem:good-trho} from Appendix~\ref{App:Aux},
the probability that \eqref{good-H} does not hold is $O(\lambda/n)$,
showing that with probability $1-4e^{-n}-O(\lambda/n)=1-O(\lambda/n)$,
\[
 \begin{aligned}
\hat\delta_2\paren{\frac 1{\rho(G)} \hat B,\HnW}
&\leq\horacle(\HnW) +
O_P\paren{\sqrt[4]{
        \lambda^2\paren{\frac{k^2}{ n^2\rho}
            + \frac{\log k}{ n\rho}
        }}}.
 \end{aligned}
\]
This holds conditioned on an event $E$ of probability $O(\lambda /
n)$.
To bound the contribution of $E$ to the overall error, we bound
$\frac 1{\rho(G)}\|\hat B\|_2\leq\frac 1{\rho(G)}\|\hat B\|_\infty\leq \lambda$
and $\|\HnW\|_2\leq\|\HnW\|_\infty\leq \lambda$,
giving an error contribution
of
$O_P(\lambda^2/n)=O_P(\sqrt[4]{\lambda^2/\rho n})$ which we can absorb into the error already present.

To prove the almost sure statement, we use that $\eps_n(W)\to 0$ almost surely, which
by Lemma~\ref{lem:good-trho} (part 2) from Appendix~\ref{App:Aux}
implies that
$\rho(\Q)/\rho\to 1$ almost surely. Since the error probability in \eqref{high-concentration}
is exponentially small, we can use the Borel-Cantelli Lemma to obtain the a.s.
statement.
\end{proof}

\subsection{Estimation of the graphon $W$}
\label{app:prop-proof-completed}

To deduce Theorem~\ref{thm:final'} from Theorem~\ref{thm:H'},
we will bound $\horacle(\HnW)$ in terms
of $\oracle(W)$, and $\delta_2(\hat B/\rho(G),W)$ in terms
of $\hat\delta_2(\hat B/\rho(G),\HnW)$.  We will show that the leading error
in both cases is an additive error of $\eps_n(W)$.
To do this, we need two  lemmas.

\begin{lemma}
\label{lem:equi-part}
Fix $n$ and $k\leq n$.
\begin{enumerate}
\item[(i)] For each equipartition $\pi:[n]\to [k]$ and each permutation $\sigma:[n]\to [n]$,
$\pi\circ\sigma$ is an equipartition.
\item[(ii)] For all equipartitions $\pi,\pi':[n]\to [k]$ there exists a permutation $\sigma:[n]\to [n]$
such that $\pi'=\pi\circ\sigma$.
\end{enumerate}
\end{lemma}

\begin{proof}
 Any equipartition must have
exactly $L_-$ classes of size $\lfloor n/k \rfloor$ and $L_+$ classes of
size $\lceil n/k \rceil$, where $L_\pm$ are determined by the equations
$L_-+L_+=k$, $L_-\lfloor n/k \rfloor +L_+\lceil n/k
\rceil=n$; and  any partition
with these properties is an equipartition. The statement follows.
\end{proof}

To state the next lemma, we define the \emph{standard equipartition} $\pi$ of
$[n]$ into $k$ classes to be the partition into the classes
$I_i=\{n_{i-1}+1,\dots,n_i\}$, $i\in [k]$,
where $n_i= \floor{ i n/k}$. Note that $n_0=0$, $n_k=n$, and
  $\floor{n/k}\leq |n_{i+1}-n_i|\leq \ceil{n/k}$.

\begin{lemma}
\label{lem:equi-part-bd}
Let $B$ be a symmetric $k\times k$ matrix with nonnegative entries,
and let $\pi$ be the standard equipartition of $[n]$ into $k$ classes.  Then
\[
\|\W{B}-\W{B_\pi}\|_2\leq \sqrt{\frac {4k}n}\|B\|_2.
\]
\end{lemma}

\begin{proof}
Let $I_1,\dots,I_k$ be adjacent intervals of length $1/k$, and for $i=1,\dots, k$,
let $J_i$ be the set of point in $x\in I_i$ such that $x\leq n_i/n$,
and let $\Delta_i=I_i\setminus J_i$.  Then
$(\W{B}-\W{B_\pi})(x,y)=0$ unless $(x,y)$ lies in one of the $3k^2$ sets
$R_{ij}^{(1)}=\Delta_i\times \Delta_j$, $R_{ij}^{(2)}=I_i\setminus\Delta_i \times \Delta_j$ or
$R_{ij}^{(3)}=\Delta_i \times I_j\setminus\Delta_j$, $i,j\in [k]$.  Taking, e.g.,
$(x,y)\in R_{ij}^{(1)}$ we have that
$|(\W{B}-\W{B_\pi})(x,y)|^2=|B_{ij}-B_{i+1,j+1}|^2\leq B_{ij}^2+B_{i+1,j+1}^2$
(note that the set $\Delta_k$ is empty, so that here we only have to consider $i,j\leq k-1$).
In a similar way, the difference in $R_{ij}^{(2)}$ is bounded
by $B_{ij}^2+B_{i+1,j}^2$, and the difference in $R_{ij}^{(3)}$ is bounded
by $B_{ij}^2+B_{i,j+1}^2$.
The total contribution of all these sets can then be bounded by
\[
\sum_{ij}B_{ij}^2 (\frac 2{nk}-\frac 1{n^2})
+\sum_{ij}B_{i,j+1}^2 \frac 1{nk}
+\sum_{ij}B_{i+1,j}^2 \frac 1{nk}
+\sum_{ij}B_{i+1,j+1}^2\frac 1{n^2}
\leq \frac 4{nk}\sum_{ij}B_{ij}^2=
\frac{4k}n\|B\|_2^2.
\]
\end{proof}

\begin{proof}[Proof of Theorem~\ref{thm:final'}]

We start by bounding $\delta_2(\hat B/\rho(G),W)$.
Let $\pi:[n]\to [k]$ be a standard equipartition,
and let $(I_1,\dots,I_n)$ be a partition of $[0,1]$ into adjacent
intervals of lengths $1/n$.
By the triangle inequality, the fact that the set
of measure preserving bijections $\pi:[0,1]\to [0,1]$
contains all bijections which just permute the intervals
$I_1,\dots,I_n$ and Lemma~\ref{lem:equi-part}
\[
\begin{aligned}
\delta_2\paren{\frac 1{\rho(G)}\hat B, W}
&\leq
\frac 1{\rho(G)}\delta_2(\W{\hat B},\W{\hat B_\pi})+
\delta_2\paren{\frac 1{\rho(G)}\W{\hat B},\W{\HnW}}
+\delta_2\paren{\W{\HnW},W}
\\
&\leq
\frac 1{\rho(G)}\|\W{\hat B}-\W{\hat B_\pi}\|_2+
\hat\delta_2\paren{\frac 1{\rho(G)}\hat B,\HnW}
+\hat\delta_2\paren{\HnW,W}
.
\end{aligned}
\]
The third term is equal to $\eps_n(W)$.  To bound the first term,
we first condition on the event \eqref{good-H}, and then use
\eqref{hatb-2-bd} together with Lemma~\ref{lem:equi-part-bd}
to conclude that conditioned on
\eqref{good-H}, with probability at least $1-4e^{-n}$,
\[
\frac 1{\rho(G)}\delta_2(\W{\hat B},\W{\hat B_\pi})\leq
O\paren{\frac{\rho}{\rho(G)}\sqrt {\frac{\lambda k}{n}}}=O\paren{\sqrt[4] {\lambda^2\frac{ k^2}{\rho n^2}}}.
\]
In view of Lemma~\ref{lem:good-trho} from Appendix~\ref{App:Aux}, the probability that this bound
does not hold is bounded by $4e^{-n}+ O(\lambda/n)=O(\lambda/n)$, so in view of the fact that
$\hat B\in \B_\mu$, which shows that $\|\hat B\|_2/\rho(G)\leq\lambda$, we see that the contribution of the
failure event is again bounded by
$O_P(\lambda^2/n)=O\paren{\lambda^2\frac k{\rho n}}=O\paren{\sqrt[4]{\lambda^2\frac k{\rho n}}}$.
All together, this proves that
\begin{equation}
\label{H-to-W-bd1}
\delta_2\paren{\frac 1{\rho(G)}\hat B, W}
\leq
\hat\delta_2\paren{\frac 1{\rho(G)}\hat B,\HnW}
+\eps_n(W)+
O_P\paren{\sqrt[4]{
        \lambda^2\paren{\frac{k^2}{ n^2\rho}
            + \frac{\log k}{ n\rho}
        }}}.
\end{equation}
The corresponding a.s. bound follows again from the fact that $\rho(\Q)\to\rho$ a.s., and the fact that
all other failure probabilities are exponentially small.

Next fix $B$ such that it is a minimizer in \eqref{eps-k}.  That implies that $B$
is obtained from $W$ by averaging over a partition of $W$ into $k$ classes, which in particular
implies that $\|B\|_2\leq\|W\|_2\leq\sqrt{\|W\|_\infty\|W\|_1}\leq \sqrt\lambda$.  Together with
Lemma~\ref{lem:equi-part-bd} this implies that there is an equipartition $\pi:[n]\to [k]$
such that
\[
\begin{aligned}
\oracle(W)&\geq \|W-\W{B}\|_2
\\
&\geq \|W-\W{B_\pi}\|_2
- \sqrt{\frac {4k}\lambda n}
\\
&\geq \hat\delta_2(B_\pi,W)- \sqrt{\frac {4k}\lambda n}
\end{aligned}
\]
Using Lemma~\ref{lem:equi-part} to express $\hat\delta_2(B,\HnW)$ as a minimum
over permutations $\sigma:[n]\to [n]$, we then bound
\[
\begin{aligned}
{\hat\eps_k^{(O)}(W)}
\leq\hat\delta_2(B,\HnW)
&=\min_{\sigma}\|B_\pi-[\HnW]^\sigma\|_2
\\
&\leq
\|\W{B_\pi} - W\|_2+\hat\delta_2(\HnW,W)
\\
&\leq
\oracle(W)+\eps_n(W)+\sqrt{\frac {4k}\lambda n},
\end{aligned}
\]
where in the first line we use
$[\HnW]^\sigma$ to denote the matrix with entries $[\HnW]_{\sigma(x),\sigma(y)}$.
Together with \eqref{H-to-W-bd1} this completes the proof of the theorem.
\end{proof}

\section{Analysis of the Private Algorithm}
\label{sec:private-analysis}

In this section we prove consistency of the private algorithms.  Our
analysis relies
on some basic results on differentially private algorithms from
previous work, which are collected in Appendix~\ref{app:private}.

Compared to the analysis of the non-private algorithms, we
need to 
control several additional
error sources which were not present for the nonprivate algorithm.
In particular, we will have to control the error between $\hat\rho$ and $\rho(G)$,
the fact that the algorithm (approximately) maximizes $\hscore(B;G)$ instead of
$Score(B;\Q)$, and the error introduced by the exponential sampling error.
The necessary bounds are given by the following lemma.  To state it, we denote
the maximal degree in $G$ by $d_{\text{max}}(G)$.

\begin{lemma}
\label{lem:private-output}
Let $(\hat\rho,\hat B)$ be the output of the randomized Algorithm~\ref{alg:main-algo}.
Then the following properties hold with probability at least $1-2e^{-n\rho\eps/16}$
with respect to the coin flips of the algorithm:

\noindent 1) \centerline{$|\rho(G)-\hat\rho|\leq\rho/4$.}

\noindent 2) If $d_{\text{max}}(G)\leq\lambda\rho/4$ and $\rho(G)\geq \rho/2$,
then
\[
\begin{aligned}
Score(\hat B;G)&\geq\max_{B\in\B_{ \mu}} 
Score(B;G)-\frac{16\lambda^2\hrho^2 (k^2+1)\log n}{n\eps}.
\end{aligned}
\]
\end{lemma}
\begin{proof}
Observing that $\Pr\{|\Lap(4/n\eps)|\geq x\}=\exp(-xn\eps/4)$, we get that
\begin{equation}
\label{hatrho-rhoG}
\Pr\paren{|\rho(G)-\hat\rho|\geq \delta\rho}=e^{-\delta n \rho \eps/4},
\end{equation}
which immediately gives (1).

To prove (2), we first use (1) and the assumptions on $\rho(G)$ and
$d_{\text{max}}(G)$
to bound
\[
\lambda\hat\rho\geq\lambda(\rho(G)-\rho/4)\geq \lambda\rho/4\geq
d_{\text{max}}(G).
\]
This implies that the extended score is equal to the original score.

We conclude the proof by using Lemma~\ref{lem:exp-mech} to show that with probability at least $1-e^{-n}\geq 1-e^{-n\rho\eps/16}$,
the exponential mechanism returns a matrix $\hat B$ such that
\[
\begin{aligned}
Score(\hat B;G)&\geq\max_{B\in\B_{\mu}} 
Score(B;G)-
{\frac{4\Delta\log(|\B_\mu|)}{\eps}.}
\end{aligned}
\]
where $\Delta=\Delta=\frac{4d \mu}{n^2}=\frac{4\lambda^2\hrho^2}{n}$.  Bounding $|\B_\mu|\leq n^{k^2}$, this completes the proof of the lemma.
\end{proof}

Theorem~\ref{thm:final} will follow from the following theorem in the same way as
Theorem~\ref{thm:final'} followed from Theorem~\ref{thm:H'}.
\begin{theorem}\label{thm:H}
  Under the assumptions of Theorem~\ref{thm:final},
\begin{equation}
\label{H-bd}
\hat\delta_2\Bigl(\frac 1{\hat \rho}\hat B,\HnW\Bigr)
\leq \horacle(\HnW) +
O_P\paren{
    \sqrt[4]{
        \frac{\lambda^2\log k}{\rho n}
            }
    +\lambda\sqrt{\frac{k^2\log n}{n\eps}}
    +\frac {{\lambda}}{n\rho\eps}
          }.
  \end{equation}
  Moreover, if we replace the assumption $n\rho\geq 6\log n$ in Theorem~\ref{thm:final} by the
stronger assumption
$n\rho\eps/\log n\to\infty$, then a.s. as $n\to\infty$,
\[
\begin{aligned}
\hat\delta_2\paren{\frac 1{\hat\rho} \hat B,\HnW}
&\leq\horacle(\HnW) +
O\paren{
    \sqrt[4]{
        \frac{\lambda^2\log k}{\rho n}
            }
    +\lambda\sqrt{\frac{k^2\log n}{n\eps}}
    +\frac {\sqrt\lambda}{n\rho\eps}
          }
+o(1).
 \end{aligned}
\]
\end{theorem}

\begin{proof}[Proof of Theorem~\ref{thm:H}]
With probability at least $1-e^{-n\rho\eps/16}$, we may assume that
the output of the private algorithm obeys the
conclusions of Lemma~\ref{lem:private-output}.
With a decrement
 in probability of at most $P_n=O(\Lambda/n)$, we then have that
\begin{equation}
\label{good-HG-hatrho}
\frac\rho2\leq\rho(G)\leq 2\rho,\qquad\frac\rho2\leq\rho(\Q)\leq 2\rho
\qquad\text{and}\qquad \frac\rho4\leq\hat\rho\leq 3\rho.
\end{equation}
Next use the assumption $\rho n\geq 6\log n$, the fact that $1\leq\Lambda\leq \lambda/8$, and
Lemma~\ref{lem:max-degree} from Appendix~\ref{App:Aux} with $\beta=\lambda/({8\Lambda})$
to show that at a decrement in probability of at most
$e^{\log n-\frac 1{24}\lambda\rho n}\leq e^{-\frac 1{48}\lambda\rho n}$,
the maximal degree in $G$ is at most $(\Lambda+\frac\lambda8)\rho\leq
\frac{\lambda\rho}4$. Lemma~\ref{lem:private-output} then allows us to use
Proposition~\ref{Cor:Concentration}
with
\[
\approxerr=\sqrt{\frac{16\lambda^2\hrho^2 (k^2+1)\log n}{n\eps}}
=O\paren{\lambda\rho\sqrt{\frac{k^2\log n}{n\eps}}}.
\]
This introduces
an additional error term
$
\frac{\approxerr}\rho=O\paren{\lambda\sqrt{\frac{k^2\log n}{n\eps}}}
$
into the bound \eqref{apply-concentration} and an extra error term
of order $O\paren{\lambda\rho\sqrt{\frac{k^2\log n}{n\eps}}}$ in the upper
bound \eqref{hatb-2-bd}, leading to the estimate that, with
probability at least $1-{ 4}e^{-n}-P_n-e^{-\frac 1{48}\lambda\rho
  n}-e^{-n\rho\eps/16} = 1- O(\Lambda / n) - e^{-\Omega(n\rho \eps)}$,
\[
\hat\delta_2\paren{\frac 1\rho \hat B,\HnW}
\leq\min_{B\in \B_\mu}\hat\delta_2\paren{\frac 1{\rho}B,\HnW} +
O\paren{
    \sqrt[4]{ { \lambda^2}
        \paren{
            \frac{k^2}{n^2\rho}
            + \frac{\log k}{n\rho}
               }
              }
    +\lambda\sqrt{\frac{k^2\log n}{n\eps}}
            }
\]
and
\[
\|\hat B\|_2=O\paren{\sqrt\lambda\rho+\lambda\rho\sqrt{\frac{k^2\log n}{n\eps}}}.
\]
From here on we proceed as in the proof of \eqref{high-concentration}, except that we now move from the
minimizer $B'$ for $\horacle(\HnW)$ to
a matrix $B\in\B_\mu$ by rounding the  entries of ${\hat\rho} B'$ down to the nearest
multiple of $1/n$.  Instead of \eqref{high-concentration}, we now obtain the bound
\begin{equation}
\hat\delta_2\paren{\frac 1{\hat\rho} \hat B,\HnW}
\leq\horacle(\HnW) +
O\paren{
  \sqrt[4]{
        \lambda^2
        \paren{
            \frac{k^2}{n^2\rho}
            + \frac{\log k}{n\rho}
               }
              }
    +\lambda\sqrt{\frac{k^2\log n}{n\eps}}
            }
+ O\paren{\sqrt\lambda\Bigl|\frac{\hat\rho}\rho-1\Bigr|}
,
\label{eq:oneofourbounds}
\end{equation}
a bound which is valid with probability at least
$1-4e^{-n}-P_n-e^{-\frac 1{48}\lambda\rho n}-e^{-n\rho\eps/16}$.
Now the fact that $|\Lap(4/n\eps)|=O_P(\frac 1{n\eps})$ and $\rho(G)=\rho(1+O_P(\sqrt{\Lambda/n}))$
implies that
\[
\Bigl|1-\frac{\hat\rho}\rho\Bigr|\sqrt\lambda=O_P\paren{\frac{\lambda}{\sqrt n}}+
O_P\paren{\frac {\sqrt\lambda}{n\rho\eps}}
=
O_P\paren{\sqrt[4]{\frac{\lambda^2}{ n}}}+ O_P\paren{\frac {\sqrt\lambda}{n\rho\eps}}.
\]
Combining this with \eqref{eq:oneofourbounds}, we obtain that with probability at least
$1-4e^{-n}-P_n-e^{-\frac 1{48}\lambda\rho n}-e^{-n\rho\eps/16}$,
\begin{equation}
\label{private-bd1}
\begin{aligned}
\hat\delta_2\paren{\frac 1{\hat\rho} \hat B,\HnW}
&\leq\horacle(\HnW) +
O_P\paren{
  \sqrt[4]{
        \lambda^2
        \paren{
            \frac{k^2}{n^2\rho}
            + \frac{\log k}{n\rho}
               }
              }
    +\lambda\sqrt{\frac{k^2\log n}{n\eps}}
    +\frac {\sqrt\lambda}{n\rho\eps}
            }
            .
 \end{aligned}
\end{equation}
The contribution of the failure event can now be bounded by
\begin{equation}\label{private-bd2}
O_P\paren{\lambda\paren{e^{-n}+{\frac\lambda
      n}+e^{-n\rho\lambda/48}+e^{-n\rho\eps/16}}}
  = O_P\paren{\frac{\lambda^2}n {+ \frac{\lambda}{n\rho\eps} }}
\end{equation}
To complete the proof of the bound in probability,
we have to add the error terms from \eqref{private-bd1} and \eqref{private-bd2}.
We can simplify the resulting expression somewhat by
first noting that  the
left hand side of Eq. \eqref{private-bd1}
is of order at most $\lambda$, which shows that
for the
bound on $\hat\delta_2$ not to be vacuous,
we need $\frac{k^2}{n}\leq \frac{k^2\log n}{n\eps}\leq 1$.
We can
  therefore drop the  term $\frac {\lambda^2k^2}{\rho n^2}$ inside the fourth
  root of \eqref{private-bd1}.
  Furthermore, by the assumption of the Theorem, $\lambda\leq \sqrt n$,
  which shows that the first term in \eqref{private-bd2}
  is $O(\sqrt[4]{\lambda^2/n})$ and can hence be absorbed into the
  error terms in \eqref{private-bd1}.
%
  This gives us the main
theorem statement.

To prove bounds which hold a.s., we note that for $n\rho\eps/\log n\to\infty$, the error probability
in \eqref{hatrho-rhoG} is summable (that is, the probability of error
$p_n$ satisfies $\sum_{n=1}^\infty p_n<\infty$) for all $\delta>0$, which together with our previous results implies
that $\hat\rho/\rho\to 1$ with probability one.  Since the probability of failure
for all other events necessary for \eqref{private-bd1} to hold is summable as well,
we get that a.s.,
\[
\begin{aligned}
\hat\delta_2\paren{\frac 1{\hat\rho} \hat B,\HnW}
&\leq\horacle(\HnW)
+ O\paren{
    \sqrt[4]{
        \frac{\lambda^2\log k}{\rho n}
            }
    +\lambda\sqrt{
        \frac{k^2\log n}{n\eps}
                }
    { +\frac {\sqrt\lambda}{n\rho\eps}}
        }
+o(1),
  \end{aligned}
\]
where again $o(1)$ is a term which goes to zero with probability one as $n\to\infty$.
\end{proof}

\begin{proof}[Proof of Theorem~\ref{thm:final}]
The proof of
Theorem~\ref{thm:final}  follows from Theorem~\ref{thm:H} in essentially same way as
Theorem~\ref{thm:final'} followed from Theorem~\ref{thm:H'}.
The only modification needed is that we now have to bound
$\frac 1{\hat \rho}\delta_2(\W{\hat B},\W{\hat B_\pi})$
instead of
$\frac 1{\rho(G)}\delta_2(\W{\hat B},\W{\hat B_\pi})$.
But this is even easier, since here we won't need to distinguish several cases.
Instead, we just use that $\hat B\in\B_\mu$ implies $\|\hat B\|_\infty\leq\lambda\hat\rho$.
With the help of Lemma~\ref{lem:equi-part-bd}, we then bound this error term by
$
\lambda\sqrt{\frac{4k}n},
$
a term which can be incorporated into the error term
$O\paren{\lambda\sqrt{\frac{k^2\log n}{n\eps}}}$.
\end{proof}

\newpage
{\small \addcontentsline{toc}{section}{References}
\bibliographystyle{abbrvnat}
\bibliography{../PrivateGraphs,../graphsprivacy}
}

\newpage
\appendix
\addcontentsline{toc}{section}{Appendix}

\section{Comparison to Nonprivate Bounds from Previous Work}
\label{sec:comparisons}

The most relevant previous works are those of \citet{WO13,Chatterjee15} and
\cite{GaoLZ14}. We provide comparisons for two types of bounded graphons:
\begin{inparaenum}\renewcommand{\labelenumi}{(\arabic{enumi})}
\item $k$-block graphons, and
\item $\alpha$-H\"older graphons.
\end{inparaenum}

\mypar{$k$-block graphons} A $k$-block graphon is a function on
$[0,1]^2$ that is constant on rectangles of the form $I_i\times I_j$,
where $I_1,...,I_k$ form a partition of the interval $[0,1]$.

In this setting, the oracle error
 $\oracle(W)=0$ and
$\eps_n(W)=O_P(\sqrt[4]{k/n})$ (see
Appendix~\ref{sec:sampling-k-block}). Our nonprivate estimator then
has asymptotic error $\sqrt[4]{\frac k n + \frac{\log k}{\rho n}
       +\frac {k^2}{\rho n^2}
}$; this is dominated by $\eps_n(W)$ as long as $k\leq \rho n$
and $\rho\geq\frac{\log k}k$.
For constant $\eps$, our
private estimator has asymptotic error at most $\sqrt[4]{\frac k n +
  \frac{\log k}{\rho n} + \frac{k^4 \log^3 n}{n^2}}$. For $k\leq(n/\log^2 n)^{1/3}$ and density $\rho\geq\frac{\log k}k$,
the error of our estimator is again dominated by the (unavoidable) error of
$\eps_n(W)=\sqrt[4]{k/n}$.

Several works analyze procedures for estimating the edge-probability
matrix $\Q$ assuming that it is (exactly) a $k$-block matrix. In the
dense case $\rho=\Omega(1)$, \citet[Theorem 1.1]{GaoLZ14} show that
the least squares estimator achieves
error $\|\frac 1 \rho\hat \Q-\frac 1 \rho \Q\|_2 = O_P(\frac {k}{n} + \sqrt{\frac{\log
    k}{n}})$. They also give a matching lower bound, which shows that
the MLE is optimal with respect to $\ell_2$ estimation of $\Q$.
\citet[Theorem 2.3]{Chatterjee15} gives a polynomial-time algorithm
with higher error $\|\frac 1 \rho\hat \Q-\frac 1 \rho \Q\|_2 = O_P(\sqrt[4]{\frac k n})$.

These bounds apply to estimating the edge-probability matrix $\Q$, but
do not apply directly to estimating an underlying block graphon
$W$. Lemma~\ref{lem:sampling-k-block} shows that $\frac 1 \rho \Q$ converges to $W$
in the $\delta_2$ metric at a rate of $O_P(\sqrt[{4}]{k/n})$. Using either
of the algorithms above for estimating $W$ gives a net error rate of
$O_P(\sqrt[{ 4}]{k/n})$ for $\delta_2$ estimation of $W$. This is the best
known nonprivate rate, and is matched by our nonprivate rate.

In the sparse case, where $\rho\to 0$ as $n\to \infty$, Wolfe and
Olhede showed under additional assumptions (roughly, that entries of
$\Q$ are bounded above and below by multiples of $\rho$) that the MLE
produces an estimate $\hat \Q$ of $\Q$ that satisfies $\|\frac 1 \rho
\hat \Q - \frac 1 \rho \Q\|_2
= O_P\paren{ \frac{k}{n}\cdot \sqrt{\frac{\log(n)}{\rho}} +
  \sqrt[4]{\frac{\log^2(1/\rho)\log(k)}{n\rho}} }$ \cite[Theorem
5.1]{WO13}\footnote{The guarantee in \cite[Theorem 5.1]{WO13} is given
in terms of KL divergence. One can convert to $\ell_2$ using the fact
that $D(p\|q) = \Theta((q-p)^2/p)$ when $q-p$ is small relative to $p$.}.
Again, one can combine these with Lemma~\ref{lem:sampling-k-block} to
get a rate of $\sqrt[4]{\frac k n} + \frac{k}{n} \sqrt{\frac{\log(n)}{\rho}} +
  \sqrt[4]{\frac{\log^2(1/\rho)\log(k)}{n\rho}} $ for estimating an
  underlying $k$-block graphon $W$. Note that when $\rho$ is small,
  any of these three terms may dominate the rate.

\mypar{H\"older-continuous graphons}
The known algorithms for
estimating continuous graphons proceed by fitting a $k$-block model to
the observed data, and arguing that this model approximates the
underlying graphon.

Our results show that if $\eps$ is
constant and $W$ is
 $\alpha$-H\"older continuous (Lipschitz continuity corresponds to
 $\alpha=1$), then the nonprivate error scales as
 $\paren{\frac{1}{n\sqrt{\rho}}}^{\frac{\alpha}{2\alpha +1}} + \sqrt[4]{\frac{\log n
   }{\rho n}} + n^{-\alpha/2}$ for an appropriate choice of $k$,
   while the private error scales as $\paren{\frac{\log
     n}{n}}^{\frac{\alpha}{2\alpha +{ 2}}} + \sqrt[4]{\frac{\log n
   }{\rho n}} + n^{-\alpha/2}$ for an appropriate choice of $k$. See  Remark~\ref{rem:Hoelder-Cont-W} for details.

In the dense case ($\rho=\Omega(1)$), \cite{GaoLZ14} show that one can
estimate a $\alpha$-H\"older continuous graphon by a $k$-block graphon
with error
$$\textstyle\delta_2(W,\hat W_{LS}) = O_P(\frac {k}{n} +
\sqrt{\frac{\log k}{n}} + k^{-\alpha}+{n^{-\alpha/2}}),$$
with the last term accounting for the difference between estimating $\Q$ and $W$.
Setting $k$ to the optimal
value of $k=1/n^{\alpha+1}$ gives a rate which except for the case $\alpha=1$
is dominated by the term $\eps_n(W)=O(n^{-\alpha/2})$.
Our nonprivate bound matches this bound for $\alpha<1/2$, and is worth for $\alpha>1/2$,
while the private one is always worth.

\citet{WO13} analyse the MLE in the sparse case, again restricting to $k$-block models. They show\footnote{We state \cite[Theorem 3.1]{WO13} for the special case where one searches over $k$-block
graphons in which all intervals have size $\Theta(k/n)$ (since
allowing nonuniformly sized blocks only makes their bounds worse), the
original graphon takes values in a range $[\lambda_a,\lambda_b]$
defined by two constants such that
$0<\lambda_a<\lambda_b$, and $k$ is polynomially
smaller than $n$.}

\begin{equation} \textstyle
\delta_2(W,\hat W_{MLE}) = O_P\paren{
\frac{k}{n}\cdot \sqrt{\frac{\log(n)}{\rho}} +
\sqrt[4]{\frac{\log^2(1/\rho)\log(k)}{n\rho}} + k^{-\alpha} + \frac{\sqrt{\log(n/\rho)}}{n^{\alpha/4}}
}
\label{eq:WObound-k}
\end{equation}

The value of $k$ that maximizes this expression
(asymptotically) can be found by
setting the first and third terms to be equal; we get $k
= \paren{n\sqrt{\rho/\log n}}^{\frac 1 {\alpha+1}}$ and a resulting
bound of
$$\textstyle
\delta_2(W,\hat W_{MLE}) = O_P\paren{
\paren{\frac 1 n \cdot \sqrt{\frac{\log
      n}{\rho}}}^{\frac{\alpha}{\alpha+1}} +
\sqrt[4]{\frac{\log^2(1/\rho)\log(n)}{n\rho}} + \frac{\sqrt{\log(n/\rho)}}{n^{\alpha/4}}
}
.$$

Next, note that $\rho$ must be $\Omega(\log^3(n)/n)$ for the middle term
to be less than 1. This means that the first term is
$O(n^{\frac{-\alpha}{2(\alpha+1)}})$. For
$\alpha\leq 1$, this is never larger than the third term. We may
therefore simplify the bound to $O_P
\paren{\sqrt[4]{\frac{\log^2(1/\rho)\log(n)}{n\rho}} + \frac{\sqrt{\log(n/\rho)}}{n^{\alpha/4}}
}$, as stated in the introduction.






\section{Background on Differential Privacy}
\label{app:private}

The notion of node-privacy defined in Section~\ref{sec:diff-P} ``composes'' well, in the sense
that privacy is preserved (albeit with slowly degrading parameters)
even when the adversary gets to see the outcome of an adaptively
chosen sequence of
differentially private algorithms run on the same data set.

\begin{lemma}[Composition,
  post-processing~\citep{MM09,DL09}]\label{lem:composition}
If an algorithm $\A $ runs $t$ randomized algorithms $\A _{1},\dots,\A _{t}$, each of which is $\eps$-differentially private, and applies an arbitrary (randomized) algorithm $g$ to their results, i.e., $\A (G) = g(\A _{1}(G),\dots,\allowbreak\A _{t}(G)),$  then $\A $ is $t\eps$-differentially private. This holds even if for each $i>1$, $\A _i$ is selected adaptively based on  $\A _1(G),\ldots,\A _{i-1}(G)$.
\end{lemma}

\mypar{Output Perturbation} \label{sec:sens} One common method for obtaining efficient differentially
private algorithms for approximating real-valued functions is based on adding a small amount
of random noise to the true answer.
A {\em Laplace} random variable with mean $0$ and standard deviation $\sqrt{2}\lambda$ has density
$h(z)=\frac{1}{2\lambda}e^{-|z|/\lambda}$. We denote it by
$\Lap(\lambda)$.

In the most basic framework for achieving differential privacy, Laplace noise is scaled according to the {\em global sensitivity} of the desired statistic $f$. This technique extends directly to graphs as long as we measure sensitivity with respect to the metric used in the definition of the corresponding variant of differential privacy. Below, we explain this (standard) framework in terms of node privacy. Let $\G$ denote the set of all graphs.

\begin{definition}[Global Sensitivity~\citep{DMNS06}]\label{def:global-sensitivity}
The $\ell_1$-global node sensitivity of a function $f: \G \rightarrow \R^p$ is:
$$\displaystyle \Delta f = \max_{G,G'\, \text{node neighbors}}\|f(G)-f(G')\|_1\,.$$
\end{definition}

For example, the edge density of an $n$-node graph has node sensitivity $2/n$, since adding or deleting a node and its adjacent edges can add or remove at most $n-1$ edges. In contrast, the number of nodes in a graph has node sensitivity $1$.

\begin{lemma} [Laplace Mechanism~\citep{DMNS06}]\label{lem:DMNS}
The algorithm $\A (G)=f(G)+ \Lap(\Delta f/\eps)^p$ (which adds i.i.d.\ noise $\Lap(\Delta f/\eps)$ to each entry of $f(G)$) is $\eps$-node-private.
\end{lemma}

Thus, we can release the number of nodes, $v(G)$, in a graph $G$ with noise of expected magnitude $1/\eps$ while satisfying node differential privacy. Given a public bound $n$ on $v(G)$, we can release the number of edges, $e(G)$, with additive noise of expected magnitude $n/\eps$.




\mypar{Exponential Mechanism}
Sensitivity plays a crucial role in
another basic design tool for differentially private algorithms,
called the \emph{exponential mechanism}.

\newcommand{\hi}{i^*}

Suppose we are given a collection of $Q$ functions, $q_1,..,q_Q$ from
$\G_n$ to $\R$, each with sensitivity at most $\Delta$. The
exponential mechanism, due to \citet{MT07}, takes a data set (in our case, a graph $G$) and aims
to output the index $\hi$ of a function in the collection which has nearly maximal value
at $G$, that is, such that $q_\hi(G) \approx \max_i q_i(G)$. The
algorithm $\A$ samples an index $i$ such that $$\Pr(\A(G)=i)
\propto \exp\paren{\frac{\eps}{2\Delta}q_i(G)}\,.$$

\begin{lemma}[Exponential Mechanism \citep{MT07}, see also {\citep[Sec. 3.4]{DworkRoth14book}}]
\label{lem:exp-mech}
The algorithm $\A$
  is $\eps$-differentially private. Moreover, with probability at
  least $1-\eta$, its output $i^*$ satisfies $$q_{i^*} (G) \geq
  \max_i\paren{q_i(G)} - \frac{2 \Delta \ln(Q/\eta)}{\eps}\,. $$
\end{lemma}

\drop{
\begin{proof}
To see why the mechansim is private, note that
  $$\Pr(\A(G)=i)= \exp\paren{\frac \eps {2\Delta} q_i(G)} /
  C_{G}\, $$
  where $C_{G} = \sum_{j \in [Q]} \exp\paren{\frac \eps
    {2\Delta} q_j(G)}$ is a normalization constant. When
  $G$ changes to a vertex neighbor $G'$, both the numerator and
  denominator of $\Pr(\A(G)=i)$ change by a factor of at most
  $\exp(\pm \eps/2)$ (since,  for every $j$, the value of
  $q_j$ varies by at most $\Delta$), and hence the probability of
  observing $i$ varies by at most $\exp(\pm\eps)$, as desired.

To see why the mechanism produces an index of a function
reasonably close to the maximum, let $OPT=\max_i q_i(G)$ denote the maximum value
attained by any of the functions on input $G$, let, $Good$ denote the
set of functions which attain that value, and let $Bad_t$ denote
the set of functions in the collection with value less than
$OPT-\frac{2\Delta}{\eps} t$
for some $t>0$. For any two functions $i\in Good$ and $j \in Bad_t$,
the ratio of their probabilities of being output by the mechanism is
at least $\exp(t)$. Hence, we have
$$\Pr(Bad_t)\leq \frac{\Pr(Bad_t)}{\Pr(Good)} \leq |Bad_t|\cdot e^{-t}
\leq Q\cdot e^{-t}\, .$$
Setting $t=\ln(\eta/Q)$, we get the desired result.
\end{proof}
}

\mypar{Lipschitz Extensions}
There are cases (and we will encounter them in this paper), where the sensitivity of a function
can only be guaranteed to be low if the graph in question has sufficiently low degrees.  In this
situation, it is useful to consider extensions of these functions from graphs obeying a certain degree bound to those
without this restriction.

  \begin{Defi}[$\G_{n,d}$  and vertex extensions]
    Let $\G_{n,d}$ denote the set of graphs with degree at most
    $d$.  Given functions $f:\G_{n,d} \to \R$ and $\hat f:\G_n\to\R$, we
  say $\hat f$ is a vertex Lipschitz extension of $f$ from $\G_{n,d}$
  to $\G_n$ if $\hat f$ agrees with $f$ on $\G_{n,d}$ and $\hat f$ has
  the same node-sensitivity  as $f$,
  that is
  $$\sup_{G,G'\in \G_n \atop \text{vertex neighbors}} |\hat f(G)- \hat
  f(G')| = \sup_{G,G'\in \G_{n,d} \atop \text{vertex neighbors}} |f(G)-f(G')|\,. $$
  \end{Defi}

\newcommand{\extensionproof}
{
      \begin{proof}[Proof of Lemma~\ref{lem:Lip-extension}]
        The existence of $\hat f$ follows from a very general result
        (e.g., \cite{McShane34,Kirszbraun1934}), which states that for
        any metric spaces $X$ and $Y$ such that $Y\subset X$, and any
        Lipschitz function $f:Y\to\R$, there exists an extension $\hat
        f:X\to\R$ with the same Lipschitz constant. The explicit,
        efficient construction of extensions for linear functions is
        due to \citet{KNRS13}. The idea is to replace $f(G)$ with the
        maximum of $f(C)$ where $C$ ranges over weighted subgraphs of
        $G$ with (weighted) degree at most $d$.  It is the value of
        the following linear program:
$$       \hat f(G)= \max f(C) \text{ such that }
\begin{cases}
  C \in [0,1]^{n\times n}
  \text{ is symmetric, and} \\
  C_{i,j} \leq A(G)_{i,j} \text{ for all } i,j \text{, and} \\
  \sum_{j \neq i} C_{i,j} \leq d \text{ for all }i \in [n]\, .
\end{cases}
$$
See \cite{KNRS13} for the analysis of this program's
properties.\end{proof}
}

We close this section with the proof of
Lemma~\ref{lem:Lip-extension}.

\extensionproof


\section{Auxiliary Bounds on Densities and Degrees}
\label{App:Aux}

\begin{lemma} \label{lem:good-trho} Let $W:[0,1]^2\to[0,\Lambda]$ be a normalized graphon,
 let $\rho\in (0,\Lambda^{-1}]$,
let $\Q=\Hn(\rho W)$, let $G=G_n(\rho W)$ and assume that $\rho n$ is bounded away from zero.
Then
\begin{enumerate}
   \item $\E \rho(\Q) =\E \rho(G) = \rho $,  $Var(\rho(\Q) )=O(\rho^2\Lambda/n)$
   and $Var(\rho(G))=O(\rho^2\Lambda/n)$, so in particular
   \[
   \Pr\{|\rho(G)-\rho|\geq\delta\rho\}=O\paren{\frac\Lambda{n\delta^2}}
   \qquad\text{and}\qquad
   \Pr\{|\rho(\Q)-\rho|\geq\delta\rho\}=O\paren{\frac\Lambda{n\delta^2}}.
   \]
   for any $\delta>0$.
\
\item
Let $\eps_n(W)=\|W-\W{\Q}\|_2$.
Then
\[
(1-\eps_n(W))\frac{n}{{n-1}}\leq\frac{\rho(\Q)}\rho\leq (1+\eps_n(W))\frac{n}{{n-1}}
.
\]
\
\item
Let $\delta\in 0<\delta<1$.  With probability at least $1-2e^{- \frac16 {\delta^2\rho(\Q) n^2}}$,
\[
1-\delta\leq\frac{\rho(G)}{\rho(\Q)}\leq 1+\delta.
\]
\end{enumerate}
\end{lemma}

\begin{proof}

1.  Clearly, $\E \rho(\Q) =\E \rho(G) = \rho $.

To bound  the variance of $\rho(\Q)$, we expand $\frac 1{\rho^2}Var(\rho(\Q))$
as a sum of $n^2(n-1)^2/4$ terms of the form
$\E[W(x_i,x_j)W(x_k,x_\ell)]-\E[W(x_i,x_j)]\E[W(x_k,x_\ell)]$
with  $i<j$ and $k<\ell$.  Observing that only those terms contribute for
which either $i=k$, $j=\ell$ or $j=k$, and bounding
$\E[W(x_i,x_j)W(x_k,x_\ell)]\leq \|W\|_2^2\leq \|W\|_\infty \|W\|_1$, we obtain
that the variance of $\frac 1\rho\rho(\Q)$ is $O(\Lambda/n)$.

To bound the variance of $\rho(G)$, we first condition on $X=(x_1,\dots,x_n)$, and bound
\[
\E[\rho^2(G)\mid X]
=\rho^2(\Q)+\frac 4{n^2(n-1)^2}\sum_{i<j}\bigl(\Q_{ij}-\Q^2_{ij}\bigr)
\leq \rho^2(\Q)+\frac 2{n(n-1)}\rho(\Q).
\]
Taking the expectation over $X$ and using the bound on the variance of $\rho(\Q)$, we obtain that
\[Var(\rho(G)) =
    O\Bigl(\frac{\rho^2\Lambda}{n}+\frac\rho{n^2}\Bigr)
    =O\Bigl(\frac{\rho^2\Lambda}{n}\Bigr)
    .
    \]
where in the last step we used the assumption that $\rho n$ is bounded away from zero.

2. Note that $\rho(\Q)=\frac n{n-1}\|\Q\|_1$.
Next, we use the triangle inequality and the fact that the $L_1$-norm is bounded by the $L_2$-norm to
see that
\[
\Bigl|\frac {\|\Q\|_1}\rho-1\Bigr|
=\Bigl|\frac {\|\Q\|_1}\rho-\|W\|_1\Bigr|
\leq \|\frac 1\rho \W{\Q}- W\|_1
\leq \|\frac 1\rho \W{\Q}- W\|_2=\eps_n(W)
\]

3. Conditioned on $X$, $S=\frac{n(n-1)}2\rho(G)$ is a sum of Bernouilli random variables with
mean $E(S)=  \frac{n(n-1)}2\rho(\Q)$.
By the multiplicative Chernov bound from Lemma~\ref{lem:chernoff-mult}, we have
that for all $\beta\leq 1$
  \[
  \Pr\Bigl\{|\rho(G) - \rho(\Q)| > \delta\rho(\Q)\, \Big |\, X\Bigr\}
  \leq 2\exp\Bigl(- \frac{\delta^2}3 \frac{n(n-1)}2\rho(\Q)\Bigr)
  \leq 2\exp\Bigl(- \frac{\delta^2}6 n^2\rho(\Q)\Bigr)
  .
  \]

\end{proof}

Next we bound the maximal degree in $G=G_n(\rho W)$.

\begin{lemma}\label{lem:max-degree}
Let $W$ be a normalized graphon with $\|W\|_\infty\leq \Lambda$,  let $0<\rho\Lambda\leq 1$, and let $\beta\geq 1$.
Then with probability at least $1-n\exp(-\beta\Lambda\rho n/3)$, the maximal degree in $G_n(\rho W)$ is bounded
by $(1+\beta)\Lambda \rho n$.
\end{lemma}

\begin{proof}
Note that the degrees in $G_n$ are stochastically
dominated by those in an Erdos-Renyi random graph where edges are chosen i.i.e. with probability $\rho \Lambda$.  In such a graph,
the degree of a given vertex $i$ is a sum of $n-1$ i.i.d Bernouilli random variables, and the standard Chernov bound implies that
for all $\beta\geq 1$
\[
\Pr\bigl(d_i\geq n\rho \Lambda (1+\beta)\bigr)
\leq exp\Bigl(-\frac\beta3 n\rho\Lambda\Bigr).
\]
Taking the union bound over all $n$ vertices in $G_n(\rho W)$ this proves the claim.
\end{proof}

\section{Convergence of the edge-probability matrix for $k$-block graphons}
\label{sec:sampling-k-block}

The sampling error $\eps_n(W)$ plays a key role in the statements of
Theorems \ref{thm:final'} and \ref{thm:final}. In some settings such
as H\"older-continuous graphons, this sampling error is dominated by
the oracle error $\oracle(W)$. For $k$-block graphons, however,
$\oracle(W)=0$ and $\eps_n(W)$ becomes more significant.

\begin{lemma}
  \label{lem:sampling-k-block}
  If $W$ is a $k$-block graphon, then $\E(\eps_n(W)^2)=O(\Lambda ^2 \sqrt{k/n})$
  and $\E(\eps_n(W))=O(\Lambda \sqrt[4]{k/n})$ where $\Lambda=\|W\|_\infty$.
\end{lemma}

\begin{proof} Fix a $k$-block graphon
  $W$, and let $p_1,...,p_k$ be the lengths of the ``blocks'', that is
  the intervals defining the block representation (so that $p_t\geq 0$
  and $\sum_t p_t=1$).  Given a sample $x_1,...,x_n$ of i.i.d. uniform
  values in $[0,1]$, let $\hat p_t$ denote the fraction of the $x_i$
  that land in each block $t$.

  Aligning $\Q = \paren{W(x_i,x_j)}_{i,j\in [n]}$ with $W$ consists of
  finding a permutation $\pi$ of $[n]$. This maps each $x_i$ to one of
  the intervals $I_1,..,I_n$ where $I_\ell =
  [\frac{\ell-1}{n},\frac{\ell}{n}]$.

  We say $x_i$ is correctly aligned if its interval $I_{\pi(i)}$ is contained in
  the block in which $x_i$  landed. For each
block $t$, we can ensure that $n\min\{ p_t,\hat p_t\} -2$ of the points
$x_i$ that landed in $t$ get aligned with $t$ (the $-2$ term accounts
for the fact that up to $1/n$ of the length at each end of the
interval does not line up exactly with one of the intervals $I_\ell$). Thus, the
number of points $x_i$ that get incorrectly aligned is at most $n\sum_t
\paren{\hat p_t - \min\{\hat p_t,p_t\}+2/n} = \frac n 2\|p - \hat p\|_1 +2k$.

Each misaligned point contributes at most $2\Lambda^2 /n$ to the total squared
error $\|W - \Q\|_2^2$, so we have
$$\hat\delta_2^2(W,\Q) \leq \Lambda^2 \paren{\|p-\hat
p\|_1 + 4k/n}.$$  Each term $|p_t-\hat p_t|$ in the $\ell_1$ norm on
the right-hand side is the
deviation of a binomial from it's mean, and has standard deviation
$\sqrt{p_t(1-p_t)/n}$. This upper bounds the expected absolute
deviation by Jensen's inequality. Thus, $\E (\hat\delta_2^2(W,\Q))\leq
\Lambda^2 \sum_t
\sqrt{p_t /n} + 4\Lambda^2 k/n$. The sum $\sum_t
\sqrt{p_t /n} $ is maximized when $p_t = 1/k$ for all $t$; it
then takes the value $k\sqrt{1/(kn)} =\sqrt{k/n}$. Hence
$\E(\hat\delta_2^2(W,\Q))\leq \Lambda^2 \sqrt{k/n}+ 4\Lambda^2 k/n \leq
5 \Lambda^2 \sqrt{k/n} $. By Jensen's inequality,
$\E(\delta_2(W,\Q))\leq  \sqrt{5} \Lambda\sqrt[4]{k/n}$, as desired.
\end{proof}

\begin{coro}
  For any graphon $W$, we have $\E(\eps_n(W)) \leq  2\oracle(W) +
 O(\sqrt[4]{k/n})$.
\end{coro}

\begin{proof}
  Fix a matrix $W$, and let $W_P$ denote the best $k$-block
  approximation to $W$ in the $L_2$ norm (that is, the minimizer of
  $\oracle(W)$).
  Given a uniform i.i.d. sample in $x_1,...,x_n$, let $\HnW$ denote the matrix
  $\paren{W(x_i,x_j)}_{i,j\in [n]}$, and let $\Hn(W_P)$ denote
  $\paren{W_P(x_i,x_j)}_{i,j\in [n]}$.
  By the triangle inequality,
  $$\eps_n(W) = \hat \delta_2(W,\Hn) \leq \underbrace{\|W-W_P\|_2}_{\oracle(W)} + \eps_n(W_P) +
  \|\HnW-\Hn(W_P)\|_2.$$

  Lemma~\ref{lem:sampling-k-block} bounds $\eps_n(W_P)$ by $\sqrt[4]{k/n}$. It
  remains to bound the last term. Squaring it, we have
  $\|\HnW-\Hn(W_P)\|_2^2 = \frac 2 {n^2} \sum_{i<j}
  |(W-W_P)(x_i,x_j)|^2$. These terms are not independent, but each
  individually has expectation $\|W-W_P\|^2=\oracle(W)^2$. By linearity of
  expectation, $\E \|\HnW-\Hn(W_P)\|_2^2 \leq \oracle(W)^2$, and hence
  $\E  \|\HnW-\Hn(W_P)\|_2 \leq \oracle(W)$.
\end{proof}

\section{Bounds for H\"older-Continuous Graphons}
\label{sec:Holder}

In this section we prove Remark~\ref{rem:Hoelder-Cont-W}.
Throughout this section we assume that $W:[0,1]^2\to [0,1]$ is $\alpha$-H\"older continuous for some $\alpha\in (0,1]$, i.e., we assume that there exists a constant $C<\infty$ such that
\[
|W(x,y)-W(x',y')|\leq C \Bigl(|x-x'|+|y-y'|\Bigr)^\alpha.
\]

\begin{lemma}
\label{lem:W-WP}
Let  $H=\HnW$, and assume that the vertices of $H$ are reordered in such a way that $x_1<x_2<\dots,x_n$.
Then
\[
\|W-\W{\Hn}\|_2=O_P(n^{-\alpha/2}).
\]
\end{lemma}
\begin{proof}
We first approximate $W$ in terms of the weighted graph $\tilde H$ with weights $(\tilde H)_{ij}=W(\bar x_i,\bar x_j)$,
where $\bar x_i=\frac i{n+1}$ is the expectation of $x_i$.  Since $|x-\bar x_i|\leq \frac 1n$ when $x\in [\frac{i-1}n,\frac in]$,
we can use the H\"older continuity of $W$ to conclude that
\[
\|W-\W{\tilde H}\|_2\leq \|W-\W{\tilde H}\|_\infty\leq C\Bigl(\frac 2n\Bigr)^\alpha.
\]
To prove the lemma, it is therefore enough to prove that
\[
\E\Bigl[\|\W{\tilde H}-\W{H}\|_2^2\Bigr]=O(n^{-\alpha }),
\]
where $\E[\cdot]$ denotes expectations with respect to the random variables $ x_1,\dots,x_n$.
Using the H\"older continuity of $W$ together with Jensen's inequality, we
bound
\[
\begin{aligned}
\E\Bigl[\|\W{\tilde \Hn}-\W{H}\|_2^2\Bigr]
&=
\frac 1{n^2}\sum_{i,j\in [n]}\E\Bigl[\bigr(
W(\bar x_i,\bar x_j)-W(x_i,x_j)\bigl)^2\Bigr]
\\
&\leq
\frac {C^2}{n^2}\sum_{i,j\in [n]}
\E\Bigl[
\bigl(|\bar x_i-x_i|+|\bar x_j-x_j|\bigr)^{2\alpha}\Bigr]
\\&\leq
\frac {C^2}{n^2}\sum_{i,j\in [n]}
\E\Bigl[
\bigl(2|\bar x_i-x_i|^2+2|\bar x_j-x_j|^2\bigr)^{\alpha}\Bigr]
\\
&\leq
\frac {C^2}{n^2}\sum_{i,j\in [n]}\Bigl(
2\E\Bigl[|\bar x_i-x_i|^2\Bigr]+ 2\E\Bigl[|\bar x_j-x_j|^2\Bigr]\Bigr)^{\alpha }
\end{aligned}
\]
Using the fact that $x_i$ has expectation $\bar x_i=\frac{i}{n+1}$ and
variance $\frac 1{n+2}\bar x_i(1-\bar x_i)\leq \frac 1{4n}$, we bound the right hand side
by $C^2 n^{-\alpha }$, completing the proof.
\end{proof}

\begin{lemma}
\label{lem:eps-k}
Let $\PP_k$ be a partition of $[0,1]$ into adjacent intervals of lengths $1/k$.  Then
\[
\oracle(W)\leq
\|W-W_{\PP_k}\|_\infty \leq C \Bigl(\frac 2k\Bigr)^\alpha.
\]
\end{lemma}
\begin{proof}
Let $\PP_k=(I_1,\dots,I_k)$.  For $(x,y)\in I_i\times I_j$, $W_{\PP_k}(x,y)$ is an average over points in $I_i\times I_j$,
implying that $|W(x,y)-W_{\PP_k}(x,y)|\leq C(2/k)^\alpha.$
\end{proof}

\section{Consistency of Multi-way cuts}\label{app:cuts}

In this appendix, we prove the following theorem
which implies Theorem~\ref{thm:cuts} by the same arguments
as those which lead
from Theorems~\ref{thm:H} and \ref{thm:H'} to  Theorems~\ref{thm:final} and \ref{thm:final'}.

\begin{theorem}\label{thm:cuts-from-H}
Let $q\geq 2$ be an integer.

(i)
Under the assumptions of Theorem~\ref{thm:final'},
\begin{equation}\label{S-non-private}
d_\infty^{\text{Haus}}(S_q(G),\hat S_q(\hat B_{\text{nonprivate}}))\leq
2\horacle(\HnW)
+O_P\paren{\sqrt[4]{
        \lambda^2\paren{\frac{k^2}{ n^2\rho}
            + \frac{\log k}{ n\rho}
        }}}.
\end{equation}

(ii)
Under the assumptions of Theorem~\ref{thm:final},
\begin{equation}\label{S-private}
d_\infty^{\text{Haus}}(S_q(G),\hat S_q(\hat B_{\text{private}}))\leq
2\horacle(\HnW)+O_p\paren{
\sqrt[4]{\frac{\lambda^2\log k}{\rho n}
}
 +\lambda\sqrt{\frac{k^2\log n}{n\eps}}+\frac {{ \lambda}}{n\rho\eps}}.
\end{equation}
\end{theorem}

Before we prove the theorem,
we start with a few bounds on the Hausdorff distance of various sets of $q$-way cuts.
First, using the  definition of the cut-distance (and the fact that the set of $q$-way cuts of a graph is invariant under relabelings), it is easy
  to see (see also \cite{BCCZ14b}) that
   whenever $G$ and $G'$ are weighted graphs on $[n]$ and $\PP$ is a partition of $[n]$, then
%
\begin{equation*}
\|G/\PP-G'/\PP\|_\infty
\leq \hat\delta_\square\Bigl(\frac 1{\|G\|_1}G,\frac 1{\|G'\|_1}G'\Bigr)
 \end{equation*}
implying that
 \begin{equation}
\label{cut-bound}
d_\infty^{\text{Haus}}(S_q(G),S_q(G'))
\leq \hat\delta_\square\Bigl(\frac 1{\|G\|_1}G,\frac 1{\|G'\|_1}G'\Bigr).
 \end{equation}
In a similar way, we have that for two graphons $W,W'$,
 \begin{equation}
\label{cut-bound-W}
d_\infty^{\text{Haus}}(\hat S_q(W),\hat S_q(W'))
\leq \delta_\square\Bigl(\frac 1{\|W\|_1}W,\frac 1{\|W'\|_1}W'\Bigr).
 \end{equation}
We will also need to compare the fractional and integer cuts, $\hat S_q(G)$ and $S_q(G')$.
To do so, one can use a simple rounding argument,
as in Theorem 5.4 and its proof from \cite{BCLSV12}.  This gives the bound%
\footnote{To translate the results from \cite{BCLSV12} into
\eqref{S-hatS-bound}, we need to take into account that in
\cite{BCLSV12}, quotients where defined with a normalization of $\frac 1{n^2}$ instead of $\frac 1{n^2\|G\|_1}$ (leading to the factor
$\frac 1{\|G\|_1}$ on the right hand side of \eqref{S-hatS-bound}),
and that Hausdorff distances were defined with respect to the $L_1$-norm
 (leading to a bound which is better by a factor $q$ than the bounds in
 \cite{BCLSV12}).}
 \begin{equation}
\label{S-hatS-bound}
d_\infty^{\text{Haus}}(S_q(G),\hat S_q(G))
\leq \frac 5{\|G\|_1\sqrt n},
 \end{equation}
valid for any weighted graph $G$ with node weights $1$ on $[n]$ and maximal edge-weight $1$.
We also note that for any weighted graph $\Q$,
\begin{equation}
\label{H-WH-cut}
\hat S_q(\Q)=\hat S_q(\W{\Q})
\end{equation}
(see
\cite[Proposition 5.3]{BCLSV12} \footnote{Note that in \cite{BCLSV12}
  the notation
for integer and fractional partitions is the reverse of the one used here and in \cite{BCCZ14b}.}).
Finally, we note that $\bigl|\|W\|_1-\|W'\|_1\bigr|\leq
\|W-W'\|_\square$. In particular,
\begin{equation}
\label{norm-cut}
\Bigl|\|G\|_1-\|G'\|_1\Bigr|\leq \hat\delta_\square(G,G')
\quad\text{and}\quad
\Bigl|\|W\|_1-\|W'\|_1\Bigr|\leq \delta_\square(W,W').
\end{equation}

\begin{proof}[Proof of Theorem~\ref{thm:cuts-from-H}]
Let $\Q=\Hn(\rho W)$ and $G=G(\Q)$.
By Lemma~\ref{lem:good-trho}, we have that $\rho(\Q)\in [\rho/2,2\rho]$
with probability at least $1-O(\Lambda/n)$. By the assumptions of the two theorems,
$\rho n\geq 2\log 2$. We apply
\cite[Lemma 7.2] {BCCZ14a} to show that
$
\hat\delta_\square(G,\Q)=\rho O\Bigl(\sqrt{\frac 1{\rho n}}\Bigr)
$
with probability at least
$1-O(\Lambda/n)$.  As a consequence, again with probability $1-O(\Lambda/n)$,

\[
\hat\delta_\square\paren{\frac 1{\|\Q\|_1}\Q,\frac 1{\|\Q\|_1}G}=O\Bigl(\sqrt{\frac 1{\rho n}}\Bigr).
\]
By \eqref{cut-bound} and \eqref{norm-cut}, this implies that with the same probability
\[
d_\infty^{\text{Haus}}(S_q(G), S_q(\Q))=
O\Bigl(\sqrt{\frac 1{\rho n}}\Bigr).
\]

Next we apply \eqref{S-hatS-bound} to the weighted graph
$\Q'=\frac 1{\|\Q\|_\infty} \Q$.  Since
$\frac 1{\|\Q'\|_1}=\frac{\|\Q\|_\infty}{\|\Q\|_1}\leq\frac{\Lambda\rho}{\|\Q\|_1}$, we conclude that with probability at least
$1-O(\Lambda/n)$,
\[
d_\infty^{\text{Haus}}(S_q(\Q), \hat S_q(\Q))=
d_\infty^{\text{Haus}}(S_q(\Q), \hat S_q(\Q'))=
O\Bigl(\sqrt{\frac 1{\rho n}}+\frac\Lambda{\sqrt n}\Bigr).
\]
Since $\|F\|_\infty\leq 1$ for all $F\in S_q(G)$ and all $F\in \hat S_q(\Q)$,
we can easily absorb the failure event, getting
\begin{equation}
\label{cuts-almost-done}
d_\infty^{\text{Haus}}(S_q(G), \hat S_q(\Q))=
O_P\Bigl(\sqrt{\frac 1{\rho n}}+\frac\Lambda{\sqrt n}\Bigr).
\end{equation}
To complete the proof, we proceed as in the proof of \eqref{H-to-W-bd1} to show that
\[
\delta_2\paren{ \HnW,\frac 1{\rho(G)}\hat B}
\leq
\hat\delta_2\paren{\frac 1{\rho(G)}\hat B,\HnW}
+
O_P\paren{\sqrt[4]{
        \lambda^2\paren{\frac{k^2}{ n^2\rho}
            + \frac{\log k}{ n\rho}
        }}}.
\]
Combined with the bound from  Theorem~\ref{thm:final'}, the fact
that the cut-norm is bounded by the $L_2$-norm,
and the fact that $\delta_\square(\HnW,\frac 1{\|\Q\|_1} \Q)
\leq |1-\rho(\Q)/\rho|=O_P(\Lambda/n)$, we conclude that

\[
\delta_\square\Bigl(\frac 1{\|\Q\|_1}\Q,\frac 1{\rho(G)}\W{\hat B}\Bigr)
\leq \horacle(\HnW)
+O_P\paren{\sqrt[4]{
        \lambda^2\paren{\frac{k^2}{ n^2\rho}
            + \frac{\log k}{ n\rho}
        }}}.
\]
Combined with \eqref{H-WH-cut}, \eqref{cut-bound-W}, \eqref{norm-cut},  and the bound \eqref{cuts-almost-done},
this proves \eqref{S-non-private}.  The proof of \eqref{S-private} is
essentially identical, except that now we use Theorem~\ref{thm:final}.
\end{proof}

\section{Useful Lemmas}

    \begin{lemma}[Multiplicative Chernoff bound]\label{lem:chernoff-mult}
      Let $X_1,X_2,...,X_n$ be independent random variables taking
      values in
      $[0,1]$, and let $X=\sum_{i=1}^nX_i$. If $\E(X)\leq \mu_0$ and
      $\beta\leq 1$,
      then $$\Pr(|X- \E(X)| \geq \beta \mu_0 ) \leq
      2\exp(-\beta^2\mu_0/3)\,.$$
      For $\beta>1$, the probability is at most $$\Pr(|X- \E(X)| \geq \beta \mu_0 ) \leq
      2\exp(-\beta\mu_0/3)\,.$$
    \end{lemma}

    \begin{proof}
      Let $\mu=\E(X)$ denotes the exact mean of $X$ (so $\mu\leq
      \mu_0$). The standard multiplicative form of the Chernoff bound
      states that for $\delta>0$ (not necessarily less than 1), we have
      $$\Pr(|X - \E X|\geq \delta \mu) \leq 2\max (e^{-\frac 1 3
      \delta^2 \mu}, e^{-\frac 1 3
      \delta \mu})\,.$$
    Setting $\delta\mu = \beta\mu_0$ (that is, $\delta =
    \frac{\beta\mu_0}{\mu}$), the bound above becomes
    $\displaystyle 2 \max (e^{-\frac 1 3
      \frac{\beta^2\mu_0^2}{\mu}}, e^{-\frac 1 3
      \beta\mu_0})\,.$ Both of these terms are bounded above by
    $2\exp(-\beta^2\mu_0/3)$: the first, since $\mu_0\leq \mu$; and the
    second, since $\beta\leq 1$.
    \end{proof}

\end{document}